\documentclass[11pt]{amsart}
\usepackage{epsfig}
\usepackage{graphicx, subfigure}
\usepackage{amsmath, amssymb,mathabx,mathrsfs}
\usepackage{color}
\usepackage[toc,page]{appendix}
\usepackage{comment}
\usepackage{hyperref}

\newtheorem{thm}{Theorem}[section]
\newtheorem{cor}[thm]{Corollary}
\newtheorem{lem}[thm]{Lemma}
\newtheorem{prop}[thm]{Proposition}

\theoremstyle{definition}
\newtheorem{defn}[thm]{Definition}
\theoremstyle{remark}
\newtheorem{rem}[thm]{Remark}
\numberwithin{equation}{section}

\newcommand{\R}{\mathbb R}
\newcommand{\T}{\mathbb T}

\newcommand{\eps}{\varepsilon}

\newcommand{\X}{\mathcal{X}}

\newcommand{\st}{{\rm s}}           
\newcommand{\un}{{\rm u}}           

\newcommand{\tM}{ \widetilde{M}}

\newcommand{\tz}{ \tilde{z}}

\newcommand{\tPhi}{\tilde{\Phi}}

\newcommand{\bF}{\mathbf{F}}
\newcommand{\bI}{\bar{I}}
\newcommand{\btheta}{\bar{\theta}}
\newcommand{\bx}{\bar{x}}
\newcommand{\by}{\bar{y}}

\def\tLambda{ {\tilde \Lambda}}
\def\tGamma{\tilde{\Gamma}}

\begin{document}
\title[Non-conservative perturbations]
{Melnikov method  for non-conservative perturbations of the  three-body problem}
\author{Marian\ Gidea$^\dag$}
\address{Yeshiva University, Department of Mathematical Sciences, New York, NY 10016, USA }
\email{Marian.Gidea@yu.edu}
\thanks{$^\dag$ Research of M.G. was partially supported by NSF grant  DMS-1515851, and the Alfred P. Sloan Foundation grant G-2016-7320.}
\author{Rafael de la Llave${^\ddag}$}
\address{School of Mathematics, Georgia Institute of Technology, Atlanta, GA 30332, USA}
\email{rafael.delallave@math.gatech.edu}
\thanks{$^\ddag$ Research of R.L. was partially supported by NSF grant DMS-1800241,
 and H2020-MCA-RISE \#734577}
 \author{Maxwell Musser$^\dag$}
\address{Yeshiva University, Department of Mathematical Sciences, New York, NY 10016, USA }
\email{maxwell.musser@yu.edu }
\thanks{$^\dag$ Research of M.M. was partially supported by NSF grant  DMS-1515851, and the Alfred P. Sloan Foundation grant G-2016-7320.}

\date{}

\begin{abstract}
We consider the planar circular restricted three-body problem (PCRTBP), as a model for the motion of a spacecraft relative to the Earth-Moon system. We focus on  the Lagrange equilibrium points $L_1$ and $L_2$. There are families of Lyapunov periodic orbits around either  $L_1$ or $L_2$, forming Lyapunov manifolds.
There also exist homoclinic orbits to the  Lyapunov manifolds around either $L_1$ or $L_2$, as well as heteroclinic orbits between the Lyapunov manifold around $L_1$  and the one around  $L_2$. The motion along the homoclinic/heteroclinic orbits  can be described via the scattering map, which gives the future asymptotic of a homoclinic orbit as a function of the past asymptotic. In contrast with the more customary Melnikov theory, we do not need to assume
that the asymptotic orbits have
a special nature (periodic, quasi-periodic, etc.).

We add a non-conservative, time-dependent perturbation, as a model for a thrust applied to the spacecraft for some duration of time, or for some other effect, such as solar radiation pressure.
We compute the first order approximation of  the perturbed scattering map, in terms of fast convergent integrals of the perturbation along homoclinic/heteroclinic orbits of the unperturbed system. As a possible application, this result can be used to determine the trajectory of the spacecraft upon using the thrust.
\end{abstract}

\maketitle
\section{Introduction}
\subsection{Motivation} A motivation for this work is the following situation from astrodynamics. Consider a spacecraft traveling  between the Earth and the Moon. Assume that the spacecraft is coasting along the stable/unstable hyperbolic invariant manifolds associated to the periodic orbits near one of the center-saddle equilibrium points, at some fixed energy level. Such an orbit is driven by the gravitational fields of the Earth and the Moon, and does not require using the thrusters.  The total energy is preserved along the orbit.  One can describe the motion of the spacecraft in terms of some geometrically defined coordinates:  an `action' coordinate describing the size of the periodic orbit associated to the stable/unstable invariant manifold,   an `angle' coordinate describing the asymptotic phase of the motion,  and  a pair of `hyperbolic' coordinates describing the position of the spacecraft relative to the corresponding stable/unstable manifold.

Suppose now that we want to make a maneuver 
in order  to jump from the hyperbolic invariant manifold on the given energy level to another hyperbolic invariant manifold  on  a different energy level. Mathematically, turning on the thrusters amounts to  adding a small, non-conservative, time-dependent perturbation to the original system. If such a perturbation is given, we would like to  estimate its effect on the orbit of the spacecraft. More precisely, we would like to compute the change in the action and angle coordinates associated to the orbit as a result of applying the perturbation. We want to obtain such  an estimate in terms of the original trajectory of the unperturbed system, and of the particular perturbation.

Other non-Hamiltonian perturbations, for instance, due to solar radiation pressure and solar wind,  can be considered.

In this paper, we will investigate the  general  problem  of adding a non-Hamiltonian perturbation to a homoclinic/heteroclinic trajectory and computing  the effect on the homoclinic/heteroclinic orbits.

Note that adding a non-Hamiltonian perturbation (e.g., a small dissipation)
may destroy all the periodic orbits. Nevertheless, the normally hyperbolic
manifold, formed by the collection of periodic orbits,
persists (see details later). In contrast with the most customary versions of
the Melnikov theory, which assume that the asymptotic orbits are periodic
or quasiperiodic,  and that they preserve their nature under perturbation,
we consider  homoclinic excursions to a normally hyperbolic manifold.
The asymptotic orbits could change their nature under the perturbations.
For example, a family of Lyapunov periodic orbits  subject to a small dissipation  could get transformed into a family of orbits that converge to a critical point.

\subsection{Brief description of the  main results and   methodology}

In the model that we consider, the family of periodic orbits about either $L_1$ or $L_2$ forms a normally hyperbolic invariant manifold (NHIM). Each NHIM has hyperbolic stable and unstable manifolds. We will assume that
the stable and unstable manifolds of the NHIM corresponding to either $L_1$ or $L_2$ intersect transversally,  and also that the unstable (stable) manifold of the NHIM corresponding to $L_1$ intersects transversally the stable (unstable) manifold of the NHIM corresponding to $L_2$. This assumption can be verified numerically for a wide range of energy levels and mass parameters in the PCRTBP (see, e.g., \cite{koon2000heteroclinic}). Thus, there exist homoclinic orbits to either one of the NHIM's, as well as heteroclinic orbits between the two NHIM's. There exist scattering maps associated to the transverse homoclinic intersections, as well as to the transverse heteroclinic intersections.

There exist some neighborhoods of $L_1$ and of $L_2$ where the Hamiltonian of the PCRTBP can be written in a normal form, via some suitable symplectic action-angle and hyperbolic coordinates $(I,\theta, y,x)$.
In particular, each NHIM can be parametrized in terms of the action-angle coordinates $(I,\theta)$.
Therefore, the scattering map can also be described in terms of these coordinates. In the unperturbed case, the scattering map is particularly simple:   it is a shift in the angle coordinate (a phase shift).

The fact that we use normal form coordinates to express the Hamiltonian, and we subsequently estimate the scattering map
in terms of the action-angle coordinates, is a matter of practical convenience. Normal forms are often used to compute  numerically, with high precision,   the periodic orbits and the  NHIM's around the equilibrium points, as well as the corresponding stable and unstable manifolds, e.g., \cite{jorba1999methodology,gomez2001dynamicsIII}.

For applications, it is important to note that the scattering map for the PCRTBP can be  computed numerically with high precision; see \cite{CanaliasDMR06,delshams2016arnold,capinski2016arnold}.

In this paper we study the effect of a small, non-Hamiltonian, time-dependent perturbation that is added to the system. Provided that the  perturbation is small enough, the NHIM's will persist \cite{Fenichel71}, although periodic orbits inside the NHIM's may disappear.
Also, the transverse homoclinic/heteroclinic orbits,  hence the scattering map, will survive in the perturbed system.

The main contribution is that we compute the effect of the perturbation on the action and angle components of  the scattering map. More precisely, we use Melnikov theory to provide  explicit estimates -- up to first order with respect to the size of the perturbation -- for  the difference between the perturbed scattering map and the unperturbed one, relative to the action and angle coordinates. The resulting expressions are given  in terms of convergent improper integrals of the perturbation evaluated along segments of homoclinic/heteroclinic orbits of the unperturbed system.
One important aspect in the computation is that, in the perturbed system, the action is a slow variable, while the angle is a fast variable.

We stress that, unlike the usual treatments of Melnikov theory, when one studies orbits homoclinic
to hyperbolic fixed points, periodic or quasi-periodic orbits, here we study orbits homoclinic
to NHIM's. The asymptotic  dynamics inside the NHIM's may change under the perturbation.

The effect of the perturbation on the action-angle components of the scattering map can be interpreted, in the context of astrodynamics,  as follows. The difference in the action coordinates between the perturbed and the unperturbed scattering map
can be interpreted as the change in energy due to the maneuver, or equivalently, the change in the `size' of the periodic orbit which the homoclinic/heteroclinic orbit is asymptotic to. The difference in the angle coordinates between the perturbed and the unperturbed scattering map can be interpreted as the change in  asymptotic phase due to the maneuver.

We mention here that there are numerous works  on using hyperbolic invariant manifolds to design low-energy space mission, see, for example  \cite{koon2000heteroclinic,gomez2001dynamics,belbruno2004capture,belbruno2010weak,parker2014low}, and the references listed there.
We hope that our results can be used to optimize the thrust that needs to be applied in order to maneuver between hyperbolic invariant manifolds on different energy levels.

\subsection{Related works}
The study of Hamiltonian systems subject to non-conservative perturbations is of practical interest in physical models, such as in celestial mechanics, where dissipation leads to migration of  satellites  and spacecrafts \cite{milani1987non,gkolias2017hamiltonian,de2019global,calleja2020kam}.

Computation of the scattering map, similar to the one in this paper, have been done in the case of the pendulum-rotator model  subject to Hamiltonian perturbations, e.g., \cite{DelshamsLS08a,gidea2018global}. The rotator-pendulum model is a product system and is  naturally endowed with action-angle and hyperbolic coordinates.  It has two conserved  quantities: the action of the rotator and the total energy of the pendulum.
The effect of the perturbation on the action component of the scattering map is relatively easy to compute directly. On the other hand, the effect on the angle component of the scattering map is more complicated to compute, since this is a fast variable. To circumvent this difficulty, the papers \cite{DelshamsLS08a,gidea2018global} use  the symplecticity of  the scattering map to estimate indirectly the effect of the perturbation on the angle component of the scattering map.

The perturbed scattering map has been computed  in the case of the pendulum-rotator model subject to non-conservative  perturbations  in\break\cite{gidea2021global}. Since the perturbations are not Hamiltonian, the symplectic  argument used in \cite{DelshamsLS08a} can no longer be applied. Therefore, to determine the effect  of the perturbation on the angle component of the scattering map, a direct computation is performed in \cite{gidea2021global}.

The PCRTBP model considered in this paper presents some significant differences from the pendulum-rotator model. First, it has only one conserved quantity, namely the total energy. Second, the PCRTBP is not a product system, and does not carry a globally defined system of action-angle and hyperbolic coordinates. Third, in the unperturbed case  the stable and unstable manifolds associated to a NHIM do not coincide.
For these reasons, we construct locally defined  systems of
action-angle and hyperbolic coordinates along  the unstable manifold as well as along  the stable manifold, respectively. At the intersection of the unstable and stable manifolds, the two coordinate systems do not agree in general. So the computation of the perturbed scattering map  has to take into account  the `mismatch' between these coordinate systems. The dynamics in these coordinate systems fails to be of product type, as there is a coupling between the action-angle and the hyperbolic coordinates, which also needs to be taken into account in the computation.
All of these features make the computation of the perturbed scattering map  for the planar circular restricted three-body problem more intricate than for the rotator-pendulum system. Some of the calculations are simplified taking advantage that some of the variables in the unperturbed system are slow variables, but
we can deal with perturbation theory for fast variables by observing that,
near the NHIM's, the difference between the variables and  their
asymptotic values is slow (a technique already used in  \cite{gidea2021global}.

\subsection{Structure of the paper}
The main result, Theorem \ref{thm:main}, is given in Section \ref{section:setup}. It
provides an expansion of the perturbed scattering map in terms of the unperturbed scattering map,  where the first order term in the expansion is given explicitly in Section \ref{section:proofs}, in Proposition \ref{prop:PCRTBP_change_in_I} and Proposition \ref{prop:PCRTBP_change_in_theta}. Section \ref{sec:pcrtbp} describes how to verify the  hypotheses of Theorem \ref{thm:main} in the context of the PCRTBP. Section \ref{section:coordinates} defines some suitable coordinate systems, which we use to describe the geometric objects of interest.  The proof of the main result is given in Section \ref{section:coordinates}.

\section{Set-up and main result}
\label{section:setup}
Consider a (real analytic) $\mathcal{C}^\omega$, symplectic manifold  $(M,\Omega)$ of dimension $(2m)$.
Each point $z\in M$ is described via a system of local coordinates $z=z(p,q)\in\R^{2m}$
such that $\Omega$ relative to these coordinates is the standard symplectic form
\begin{equation}\label{eqn:standard_symplectic}
\Omega=dp\wedge dq=\sum_{i=1}^{m}dp_i\wedge dq_i. \end{equation}

On $M$ we consider a non-autonomous system of differential equations
\begin{equation} \label{eqn:generalperturbation}
\frac{d}{dt} z = \X^0(z) + \eps\X^1(z,t;\eps),
\end{equation}
where $\X^0:M\to TM$ is a $\mathcal{C}^\omega$-smooth vector field on $M$, $\X^1:M\times \R\times (-\eps_0,\eps_0)\to TM$ is a time-dependent, parameter dependent $\mathcal{C}^r$-smooth vector field on $M$,  with   $r$ sufficiently large,
and $\epsilon\in\R$ is a `smallness' parameter,  taking values in the  interval $(-\eps_0,\eps_0)$ around $0$.
The  dependence of $\X^1(z,t;\eps)$ on the time $t$ is assumed to be of a general type -- not necessarily periodic or quasi-periodic.

The flow of \eqref{eqn:generalperturbation} will be denoted by $\Phi^t_\eps$.

The assumption that $\X^0$ is analytic is motivated by applications to celestial mechanics.   We will consider the case when $\X^0$ is Hamiltonian vector field for the PCRTBP. We will use this assumption only to be
able to quote several normal form theorems.  We believe it
could be weakened to finite differentiability at the price of providing some
new normal form theorems.

The assumption that $\X^1$ is only $\mathcal{C}^r$ and possibly not Hamiltonian is also motivated by applications, as $\X^1$ can model the thrust applied to a spacecraft for some time. In particular $\X^1$ can have compact support in space and in time. Note that even if the perturbation were analytic, the NHIM's which play a role in our
treatment can only be assumed to be finite differentiable. The regularity is
limited by ratios between the tangential and normal contraction rates,
as well as by the regularity of the perturbation. Since the motion on
the unperturbed manifold is integrable, we have that for $\eps$ small enough,
the tangential rates are close to zero, so that the limitations of
regularity due to the rates become irrelevant.

Below, we will require that the  vector fields $\X^0$, $\X^1$ satisfy additional assumptions.

\subsection{The unperturbed system}\label{section:unperturbed}
We assume that the vector field $\X^0$ represents an autonomous Hamiltonian vector field, that is, $\X^0=J\nabla H_0$ for some $\mathcal{C}^\omega$  Hamiltonian function $H_0:M\to \R$, where $J$ is an almost complex structure compatible with the standard symplectic form given by \eqref{eqn:standard_symplectic}, and the gradient $\nabla:=\nabla_z$ is with respect to the associated Riemannian metric. The Hamilton equation  for the unperturbed system is:
\begin{equation}\label{eqn:unperturbed}
\frac{d}{dt} z=J\nabla H_0(z).
\end{equation}

\subsubsection{Homoclinic connections}

We assume that the Hamiltonian flow associated to $H_0$ satisfies the conditions \textbf{(A-i)}, \textbf{(A-ii)} and \textbf{(A-iii)}, below.
In Section \ref{sec:pcrtbp}, we will see that these conditions can be verified in the planar circular restricted three-body problem.

\begin{itemize}
\item[\textbf{(A-i)}] \emph {There exists an equilibrium point $L$ of saddle-center type, that is, the linearized system $DJ\nabla H_0$ at $L$  has eigenvalues of the type $\pm \lambda, \pm i\omega$, with $\lambda,\omega\neq 0$.}
\end{itemize}

Consequently, by the Lyapunov Center Theorem \cite{moser1958generalization}, there exists a $1$-parameter family of periodic  orbits $\lambda_0(h)$, parametrized by the energy level $H=h$, for $h\in(H(L),h_1)$ with $h_1$ sufficiently small, such that $\lambda_0(h)$  shrinks to $L$ as $h\to H(L)$.
This family of periodic orbits determines a  $2$-dimensional manifold  $\Lambda_0\simeq D\times\mathbb{T}\subseteq M$  which is a   normally hyperbolic invariant manifold  (NHIM)  with boundary for the Hamiltonian flow $\Phi^t_0$ of $H_0$, where $D$ is closed interval with non-empty interior contained in $(H(L),h_1)$. The notion of a NHIM is recalled in Definition \ref{defn:NHIM}, Appendix \ref{sec:NHIM}.

The NHIM $\Lambda_0$ is symplectic when endowed with the form $\Omega_{\mid\Lambda_0}$, where $\Omega$ is given by \eqref{eqn:standard_symplectic}. Moreover, $\Lambda_0$ is foliated by the periodic orbits $\lambda_0(h)$, i.e., $\Lambda_0=\bigcup_{h\in D}\lambda_0(h)$. The flow $\Phi^t_0$ on each $\lambda_0(h)$ is a constant
speed flow. In particular, the dynamics restricted to the NHIM is integrable. Therefore $\Lambda_0$ can be parametrized in terms
of symplectic action-angle variables $(I,\theta)$, so that each periodic orbit
$\lambda_0(h)$ represents a level set $I_h$ of the action variable. The action $I_h$ is uniquely determined by the energy level $H_0=h$.
In fact, as we will see in Section \ref{sec:coordinates}, there exist a  system of action-angle and hyperbolic variables  $(I,\theta,y,x)$ in a neighborhood of $L$ such the Hamiltonian $H_0$ can be written in a normal form.

\begin{itemize}
\item [\textbf{(A-ii)}] \emph{There exists a relatively compact open   set $\mathscr{K}$ in $M$ such that the   unstable manifold $W^\un_{\mathscr{K}}(\Lambda_0)$ and the  stable manifold $W^\st_{\mathscr{K}}(\Lambda_0)$ inside $\mathscr{K}$ intersect transversally along a   \emph{homoclinic channel} $\Gamma_0\subset \mathscr{K}$.}
\end{itemize}

The definition of a homoclinic channel is given in Appendix \ref{section:scattering_review}, Definition \ref{def:homoclinic_channel}.


As a consequence of \textbf{(A-ii)}, there exist transverse homoclinic orbits to $\Lambda_0$.
The unstable and stable manifolds of $\lambda_0(h)$ are contained in the same energy level as $\lambda_0(h)$, i.e.,
$W^\un(\lambda_0(h)), W^\st(\lambda_0(h))\subseteq M_h$, where $M_h=\{z\in M\,|\, H=h\}$.
Hence, each homoclinic orbit to $\Lambda_0$ is asymptotic, in both forward and backwards time, to the periodic orbit $\lambda_0(h)$ corresponding to its  energy level $h$.

In Section \ref{sec:coordinates}, we will see that the normal form coordinates $(I,\theta,y,x)$ can be extended via the flow $\Phi^t_0$ along neighborhoods of $W^\un(\Lambda_0)$ and $W^\st(\Lambda_0)$, yielding two systems of action-angle and hyperbolic variables  $(I^\un,\theta^\un,y^\un,x^\un)$, $(I^\st,\theta^\st,y^\st,x^\st)$, respectively.
Relative to these two systems of coordinates
$W^\un(\Lambda_0)$ can be described  locally  by $y^\un=0$, and
$W^\st(\Lambda_0)$ can be described  locally by  $x^\st=0$.
The two  coordinate systems $(I^\un,\theta^\un,y^\un,x^\un)$ and $(I^\st,\theta^\st,y^\st,x^\st)$ do not agree with one another in a neighborhood of the homoclinic channel $\Gamma_0$.

In addition, we require that $H_0$  satisfies a non-degeneracy condition \textbf{(A-iii)}, which is formulated in terms
of the normal form  from Section \ref{sec:coordinates}.
Condition \textbf{(A-iii)} consists of two parts: \textbf{(A-iii-a)}, and \textbf{(A-iii-b)}.
Condition  \textbf{(A-iii-a)}  will be given in Section \ref{sec:coordinates}, and
condition  \textbf{(A-iii-b)}  will be given in Section \ref{sec:scattering_PCRTBP}.
These are explicit and verifiable conditions that the derivatives of certain functions are non-zero.

\subsubsection{Unperturbed scattering map associated to a homoclinic channel}
Let $\Omega^{-}:W^\un(\Lambda_0)\to\Lambda_0$ be the projection mapping defined by
$\Omega^{-}(z_0)=z^-_0$, where $z_0^-\in\Lambda_0$ is the footpoint of the unstable fiber through $z_0\in W^\un(\Lambda_0)$.
Similarly, let $\Omega^{+}:W^\st(\Lambda_0)\to\Lambda_0$ be the projection mapping defined by
$\Omega^{+}(z_0)=z^+_0$, where $z_0^+\in\Lambda_0$ is the footpoint of the stable fiber through $z_0\in W^\st(\Lambda_0)$.
The stable and unstable fibers are defined in Appendix \ref{sec:NHIM}, equation \eqref{stablemanif-flow}.

Consider the homoclinic channel $\Gamma_0$ from condition \textbf{(A-ii)}.
By the definition of a homoclinic
channel, $\Omega^\pm$  restricted to $\Gamma_0$ is a diffeomorphism onto its image.
To  any  homoclinic channel we can associate a \emph{scattering  map},
which is defined in Appendix \ref{section:scattering_review}, Definition \ref{def:scattering map}.
Specifically, the scattering map $\sigma_0$ associated to $\Gamma_0$ is given by:
\begin{equation}\begin{split}
&\sigma_0:\Omega^-(\Gamma_0)\subseteq \Lambda_0\to \Omega^+(\Gamma_0)\subseteq \Lambda_0,\\
&\sigma_0(z_0^-)=z_0^+,
\end{split}\end{equation}
provided that there exists a homoclinic point $z_0\in \Gamma_0$ such that $\Omega^{-}(z_0)=z^-_0$ and $\Omega^{+}(z_0)=z^+_0$.
For more details on the scattering map, see Appendix \ref{section:scattering_review}.

The energy preservations along the stable and unstable manifolds of periodic orbits implies that $\sigma_0$
leaves each periodic orbit $\lambda_0(h)$ invariant, that is,  $\sigma_0(\lambda_0(h))\subseteq \lambda_0(h)$.

Fixing a point $z_0\in \Gamma$, we have that $\sigma_0(z_0^-)=z_0^+=z_0^-+\Delta$, for some  $\Delta$ depending on $z_0$. The invariance property of the scattering map \eqref{eqn:invariant}, and the fact that $\Phi_0^t$ restricted to
$\lambda_0(h)$ is linear implies that
 \[\sigma _0(\Phi^t_0(z_0^-))=\Phi^t(z_0^+)=\Phi^t(z_0^-+\Delta)=\Phi^t(z_0^-)+\Delta,\]
for all $t$ for which $\Phi^t_0(z_0)$ remains in $\Gamma_0$.
This implies that, in terms of the action-angle coordinates $(I,\theta)$, $\sigma_0$ is given by a shift in the angle
\[\sigma_0(I,\theta)=(I,\theta+\Delta(I)),\]
for some function $\Delta $ that depends differentiably on $I$. We stress that $\Delta$ also depends on the choice of the homoclinic channel  $\Gamma_0$.


\subsection{The perturbation.}
The vector field  $\X^1$  is a time-dependent, parameter-dependent vector field on $M$.
\begin{itemize}
\item [\textbf{(A-iv)}] \emph{We assume that $\X^1=\X^1(z,t;\eps)$ is
uniformly $\mathcal{C}^r$-differentiable in all variables on $\mathscr{K}\times\R\times (-\eps_0,\eps_0)$, where the set $\mathscr{K}$ is as in the condition \textbf{(A-ii)}.}
\end{itemize}

Above, we assume that $r$ is suitably large.
We will not assume that $\X^1$ is Hamiltonian.
Thus, our setting can be used to model dissipation or forcing applied to a Hamiltonian system.
Note that non-Hamiltonian perturbations are very singular, in the sense that periodic and homoclinic orbits may disappear. On the other hand, the NHIM's
and their stable and unstable manifolds persist and can be used as the basis for
perturbative calculations.

As a particular case, we will also write our results for the case when the perturbation $\X^1$  in \eqref{eqn:generalperturbation} is Hamiltonian, that is, it is given by
\begin{equation}\label{eqn:h}
   \X^1(z,t;\eps)=J\nabla_z H_1(z,t;\eps),
\end{equation}
where $H_1$ is a time-dependent, parameter-dependent $\mathcal{C}^r$-smooth Hamiltonian function on $M$.

When the perturbation is added to the system, as we will see in Section \ref{sec:perturbed_NHIM},
the NHIM  for the unperturbed system can be continued to a NHIM  for the perturbed system, and
the transverse homoclinic/heteroclinic orbits for the unperturbed system can be continued  to
transverse homoclinic/heteroclinic orbits for the perturbed system,
provided that the perturbation is sufficiently small.
Hence there exists an associated scattering map for the perturbed system.

The goal is to quantify the effect of the perturbation  on the corresponding scattering map.

\subsection{Extended system}\label{sec:extended}
We consider the system \eqref{eqn:generalperturbation}
under the  conditions \textbf{(A-i)}, \textbf{(A-ii)},  and \textbf{(A-iii)}.
We associate to it the extended system
\begin{equation}\label{eqn:generalperturbation_t}
\begin{split}
& \frac{d}{d\tau} z= \X^0(z) + \eps\X^1(z,t;\eps), \\
& \frac{d}{d\tau} t  = 1,\\
\end{split}
\end{equation}
which is defined on the extended phase space $\tM=M\times \R$. We denote $\tz=(z,t)\in \tM$.
The independent variable will be denoted by $\tau$ from now on.
We will denote by $\tPhi^\tau_\eps$ the extended flow of \eqref{eqn:generalperturbation_t}. We have
\[ \tPhi^\tau_\eps(z,t)=( {\Phi}^\tau_\eps(z), t+\tau).\]

In the extended phase space we have the following:

\[\tLambda_0=\Lambda_0\times\R\]
is a  NHIM with boundary for the extended unperturbed flow $\tPhi^\tau_0$.

\[\tilde{\Gamma}_0=\Gamma_0\times \R\]
is  a homoclinic  channel for $\tPhi^\tau_0$.

The  scattering map associated to $\tilde{\Gamma}_0$ is given by
\[\tilde{\sigma}_0(I,\theta,t)=(I,\theta+\Delta(I),t).\]

\subsection{Main result}
\begin{thm}\label{thm:main}
\label{eqn:main}
Consider the system \eqref{eqn:generalperturbation}.

Assume that the unperturbed system  $\X^0$  satisfies  the conditions \textbf{(A-i)}, \textbf{(A-ii)} and \textbf{(A-iii)}, and that the perturbation $\X^1$  satisfies the condition  \textbf{(A-iv)}.

Then, there exists $\eps_1$, with $0<\eps_1<\eps_0$, such that the following hold true:
\begin{itemize}
\item[(i)]  For all $\eps\in(-\eps_1,\eps_1)$,
there is a  $\mathcal{C}^{\ell}$-family of NHIMs $\tilde\Lambda_\eps$
for the extended  flow
$\tilde{\Phi}^t_\eps$,
for some $\ell\geq 1$, which form a continuation of the NHIM $\tilde\Lambda_0$
of $\tilde{\Phi}^t_0$;
\item [(ii)] For all $\eps\in(-\eps_1,\eps_1)$, the unstable and stable manifolds  of $\tilde\Lambda_\eps$, $W^\un(\tilde\Lambda_\eps)$ and $W^\st(\tilde\Lambda_\eps)$, respectively,  intersect transversally,  in the extended  phase space $\tM$, along a homoclinic channel $\tilde\Gamma_\eps$;
\item[(iii)] The perturbed scattering map $\tilde{\sigma}_\eps$ associated to $\tilde\Gamma_\eps$ can be written as
\begin{equation}\label{eqn:main_sigma}
\tilde{\sigma}_\eps(I,\theta,t)=\tilde{\sigma}_0(I,\theta,t)+\eps \tilde{\mathcal{S}}(I,\theta,t)+O_{\mathcal{C}^{\ell}}(\eps^{2}),
\end{equation}
where $\tilde{\mathcal{S}}=(\tilde{\mathcal{S}}^I,\tilde{\mathcal{S}}^\theta,\tilde{\mathcal{S}}^t)$ is a mapping from some domain in $D\times \T\times\R$ to its image in $\R\times \T\times\R$  as follows:  \begin{itemize}
\item [(iii-a)] the components $\tilde{\mathcal{S}}^I$ and $\tilde{\mathcal{S}}^\theta$
are given by
\eqref{eqn:PCRTBP_I_plus_minus} and \eqref{eqn:PCRTBP_change_in_theta}, respectively, and
\item [(iii-b)] the component $\tilde{\mathcal{S}}^t$ is given by $\tilde{\mathcal{S}}^t(I,\theta,t)=t$.
\end{itemize}
\end{itemize}
\end{thm}

We recall the notation $O_{\mathcal{C}^k}(\cdot )$ used above:  $f=O_{\mathcal{C}^k}(g)$ means that $\|f\|_{\mathcal{C}^k}\leq M\|g\|_{\mathcal{C}^{k}}$ for some  $M>0$, where $k\geq 0$,  and $\|\cdot\|_{\mathcal{C}^k}$ is the ${\mathcal{C}^k}$-norm.
In the sequel, to simplify the notation we will write $O(\cdot )$ without the subscript indicating the function space topology, whenever this can be inferred from the context.

\subsection{Heteroclinic connections}\label{sec:heteroclinic}
Instead of the conditions \textbf{(A-i)}, \textbf{(A-ii)}, \textbf{(A-iii)}  we  assume that the Hamiltonian flow associated to $H_0$ satisfies the conditions  \textbf{(A'-i)}, \textbf{(A'-ii)}, \textbf{(A'-iii)}
from below.

Condition  \textbf{(A'-i)} has two parts \textbf{(A'-i-a)} and \textbf{(A'-i-b)}.

\begin{itemize}
\item[\textbf{(A'-i-a)}] \emph {There exists two equilibrium points $L^1$, $L^2$ of saddle-center type.}

\end{itemize}
We do not assume that the two equilibrium points are on the same energy level, that is $H(L^1)\neq H(L^2)$ in general.
Consequently, for each equilibrium point $L^1$, $L^2$ there exists a $1$-parameter family of closed orbits $\lambda^1_0(h)$, for $h\in D^1$, and  $\lambda^2_0(h)$ for $h\in D^2$, where $D^1$, $D^2$ are some closed intervals contained in some neighborhoods of $H(L^1)$, $H(L^2)$, respectively.

\begin{itemize}
\item[\textbf{(A'-i-b)}] \emph {There exists an interval of energies $D\subseteq D^1\cap D^2$, with non-empty interior, such that there are periodic orbits  $\lambda^1_0(h)$, $\lambda^2_0(h)$  for all $h\in D$. Moreover, there exist normal form coordinates $(I^1,\theta^1,y^1,x^1)$ and $(I^2,\theta^2,y^2,x^2)$ defined around $\lambda^1_0(h)$, $\lambda^2_0(h)$, respectively,  for all $h\in D$. These normal form coordinates are as in Section
    \ref{sec:coordinates}.}
\end{itemize}

This condition implies that each family of periodic orbits determines a  $2$-dimensional NHIM  $\Lambda^1_0=\bigcup_{h\in D}\lambda_0(h)$, $\Lambda^2_0=\bigcup_{h\in D}\lambda^2_0(h)$ in  $M$. Moreover, there are neighborhood of $\Lambda^1_0$ and $\Lambda^2_0$ where  $H_0$ can be written in a normal form as in  Section
    \ref{sec:coordinates}.

\begin{itemize}
\item [\textbf{(A'-ii)}] \emph{There exists a relatively compact open   set $\mathscr{K}$ in $M$ such that the   unstable manifold $W^\un_{\mathscr{K}}(\Lambda^1_0)$ and the  stable manifold $W^\st_{\mathscr{K}}(\Lambda^2_0)$ inside $\mathscr{K}$ intersect transversally along a   \emph{heteroclinic channel} $\Gamma_0\subset \mathscr{K}$.}
\end{itemize}

The definition of a heteroclinic channel is given in Definition \ref{def:heteroclinic_channel}, Appendix \ref{section:scattering_review}.
As a consequence, there exist transverse heteroclinic orbits from $\Lambda^1_0$ to $\Lambda^2_0$.
Each such heteroclinic orbit is asymptotic in backwards time  to a periodic orbit $\lambda^1_0(h)$, and is asymptotic in forward time to a periodic orbit $\lambda^2_0(h)$.

As in the case of homoclinic connections, we require some non-degeneracy condition \textbf{(A'-iii)}, formulated in terms of normal forms, which is the analogue of condition \textbf{(A-iii)} in Section \ref{section:setup}.
For the sake of brevity, we will not formulate this condition explicitly.


\subsubsection{Unperturbed scattering map associated to a heteroclinic channel}

As before, we define $\Omega^{-,1}:W^\un(\Lambda^1_0)\to\Lambda^1_0$ by
$\Omega^{-,1}(z_0)=z^-_0$, where $z_0^-\in\Lambda^1_0$ is the footpoint of the unstable fiber through $z_0\in W^\un(\Lambda^1_0)$, and
 $\Omega^{+,2}:W^\st(\Lambda^2_0)\to\Lambda^2_0$  by
$\Omega^{+,2}(z_0)=z^+_0$, where $z_0^+\in\Lambda^2_0$ is the footpoint of the stable fiber through $z_0 \in W^\st(\Lambda^2_0)$.

Associated to the heteroclinic channel $\Gamma_0$   we can define the scattering  map as in Definition \ref{def:scattering map_heteroclinic} in Appendix \ref{section:scattering_review}.
Specifically,
\[\sigma_0:\Omega^{-,1}(\Gamma_0)\subseteq \Lambda^1_0\to \Omega^{+,2}(\Gamma_0)\subseteq \Lambda^2_0,\]
is given by
\[\sigma_0(z_0^-)=z_0^+,\]
provided that there exists a $z_0\in \Gamma_0$ such that $\Omega^{-,1}(z_0)=z^-_0$ and $\Omega^{+,2}(z_0)=z^+_0$.

Since $H_0$ is constant along heteroclinic orbits, we have that $I(z^-_0)=I(z^+_0)$.
Then the scattering map, expressed in terms of the action-angle coordinates, is given by
\[\sigma_0(I,\theta)=(I,\theta+\Delta(I)),\]
for some function $\Delta$ that depends smoothly on $I$.

In this case, we can obtain a result similar to Theorem \ref{thm:main}. For brevity, we will not provide the precise formulas for the components of the corresponding expansion of the perturbed scattering map,  which is analogous to \eqref{eqn:main_sigma}.
Such formulas are analogues of  \eqref{eqn:PCRTBP_I_plus_minus} and \eqref{eqn:PCRTBP_change_in_theta}.

\section{Geometric structures in the planar circular restricted three-body problem.}
\label{sec:pcrtbp}
In this section we survey the
status of the verification of the conditions \textbf{(A-i)}, \textbf{(A-ii)}, \textbf{(A-iii)} from Section \ref{section:unperturbed}, or the conditions  \textbf{(A'-i)}, \textbf{(A'-ii)}, \textbf{(A'-iii)} from Section \ref{sec:heteroclinic}, in the concrete model of the PCRTBP.  Some of the verifications
in the literature are rigorous and some of them are numerical.

Of course, the verification of the hypothesis of Theorem~\ref{thm:main} in a
concrete model  does not affect the validity of the rigorous arguments
establishing Theorem~\ref{thm:main},
and the reader  interested only in rigorous results may safely skip
this section.

We note that, since our hypothesis are mainly transversality
conditions, they can be verified with finite precision
calculations, which seem to be safe calculations for today's standard
and could well be accessible to \emph{``computer assisted proofs''}.
We hope that this work could stimulate more extensive verifications.

The PCRTBP is a model describing
the motion of an infinitesimal body under the Newtonian gravity exerted by two heavy bodies   moving  on circular orbits about their center of mass,  under the assumption that these orbits are not affected by the gravity of the infinitesimal body.

We can think of the heavy bodies (referred to as primaries) representing the Earth and Moon, and the infinitesimal mass representing  a spaceship.

It is convenient to study the motion of the infinitesimal body relative to a co-rotating  frame which rotates with the primaries around the center of mass, and to  use normalized units. Henceforth, the masses of the heavy bodies are  $1-\mu$ and $\mu$, where $\mu\in(0,1/2]$.
Relative to the co-rotating frame, the heavier mass $1-\mu$ is located at $(\mu,0)$, and the lighter mass $\mu$ is located at $(-1+\mu,0)$.
The motion of the infinitesimal body relative to the co-rotating frame is described via the autonomous Hamiltonian
\begin{equation}
H_0(p_1,p_2,q_1,q_2)=\frac{(p_{1}+q_{2})^{2}+(p_{2}-q_{1})^{2}}{2}-V(q_1,q_2),
\label{eqn:PRCTBP}
\end{equation}
where $\left( p,q\right) =\left(p_{1},p_{2},q_{1},q_{2}\right)\in\R^4$ represents the  momenta and  the
coordinates  of the infinitesimal body with respect to the co-rotating frame,
\begin{equation}\begin{split}
V(q_1,q_2) & =\frac{q_1^2+q_{2}^2}{2}+\frac{1-\mu}{r_{1}}+\frac{\mu}{r_{2}},\\
r_{1} & =\left((q_1-\mu)^{2}+q_2^{2}\right)^{1/2},\\
r_{2} &=\left((q_1+1-\mu)^{2}+q_2^{2}\right)^{1/2}.
\end{split}
\end{equation}

Above $V(q_1,q_2)$ represents the effective potential,
and $r_{1}(t)$, $r_{2}(t)$ represent the distances from the infinitesimal body to the masses $1-\mu$ and $\mu$,
respectively. The phase space \[M=\{(p,q)\in\R^4\,|\,q\neq (\mu,0),\textrm{ and } q\neq (-1+\mu,0)\}\] is endowed with the symplectic form \[\Omega=dp_1\wedge dq_1+dp_2\wedge dq_2.\] Note that the phase space $M$ is not compact.

The equations of motion of the infinitesimal body are given by the Hamilton equations corresponding to \eqref{eqn:PRCTBP}, that is
\begin{equation}\label{eqn:PRCTBP_equations}
\frac{d}{dt}{z}=J\nabla H_0(z),
\end{equation}
where $z=z(p_1,p_2,q_1,q_2)$ and $J$ represents the standard almost complex structure.

The Hamiltonian $H_0$ is an integral of motion, so the flow $\Phi^t_0$ of \eqref{eqn:PRCTBP_equations} leaves invariant each energy hyper-surface
\begin{equation}\label{eqn:energy}
M_h:=H_0^{-1}(h)=\{(p_1,p_2,q_1,q_2)\in M\,|\,H_0(p_1,p_2,q_1,q_2)=h\}.
\end{equation}

The system has three equilibrium points $L_{1},L_{2},L_{3}$  located on the $q_1$-axis,
and two other equilibrium points $L_4$, $L_5$, each lying in the $(q_1,q_2)$-plane and
forming an equilateral triangle with the primaries. These are known as the  Lagrange points.
In our convention, $L_1$ is located between the primaries, $L_2$ is on the side of the lighter primary, and $L_3$ is  on the side of the heavier primary.
The linearized stability of $L_{1},L_{2},L_{3}$ is of center-saddle type.
The linearized stability of $L_{4}$, $L_{5}$ is of center-center type, for $\mu$ less than some critical value $\mu_{\textrm{cr}}$.

We note that condition \textbf{(A-i)} is satisfied for each of the equilibrium points $L_{1},L_{2},L_{3}$.

For $i=1,2,3$, for each energy level $h$, with $H(L_i)<h$ and $h$ sufficiently close to $H(L_{i})$, there exists a unique periodic orbit $\lambda_0(h)$ near the equilibrium point
$L_{i}$, which is referred to as a Lyapunov orbit. The existence of such periodic orbits
follows from the Lyapunov Center Theorem (see, e.g., \cite{moser1958generalization}).
Moreover, there exists a neighborhood of $L_{i}$ in the phase space where the Hamiltonian  $H_0$ can be written in a normal form relative to some suitable coordinates $(I,\theta,y,x)$; see Section \ref{sec:coordinates}.

Each Lyapunov orbit is hyperbolic in the energy surface, so it has associated $2$-dimensional unstable and stable manifolds denoted $W^\un(\lambda_0(h))$ and $W^\st(\lambda_0(h))$.

Numerical evidence shows that these periodic orbits can be continued for energy levels $h>H(L_i)$ that are not necessarily close to $H(L_i)$ (see, e.g., \cite{broucke1968periodic}).

\textsl{Normally hyperbolic invariant manifold for the unperturbed system.}
For an energy range $h\in D$ sufficiently close to the energy of $L_i$, the family of Lyapunov orbits
\begin{equation}\label{eqn:PCRTBP_NHIM}
\Lambda_0=\bigcup_{h\in D} \lambda_0(h),
\end{equation}
defines a $2$-dimensional NHIM with boundary for the Hamiltonian flow  of \eqref{eqn:PRCTBP_equations}. The NHIM carries the symplectic structure  $\Omega_{\mid\Lambda_0}$, and it can be described in terms of the action-angle coordinates $(I,\theta)$. The action $I=I_h$ is uniquely defined by the  energy $h$, and  $\theta$ is symplectically conjugate with $I$ with respect to   $\Omega_{\mid\Lambda_0}$.
The variable  $I$ is a first integral along the trajectories of the flow on $\Lambda_0$, and the action level sets are in fact the Lyapunov orbits $\lambda_0(h)$. The motion restricted to each Lyapunov orbit is a rigid rotation in the variable $\theta$, with the frequency depending on the energy level.

The NHIM $\Lambda_0$ and its unstable and stable manifolds $W^\un(\Lambda_0)$ and  $W^\st(\Lambda_0)$ have simple descriptions in terms of the normal form coordinates $(I,\theta,y,x)$ in a neighborhood of $L_1$:
$\Lambda_0$ corresponds to $x=y=0$, $W^\un(\Lambda_0)$ corresponds to $y=0$, and $W^\st(\Lambda_0)$ corresponds to $x=0$.

\textsl{Homoclinic connections.}
We first focus on the dynamics around the equilibrium point $L_1$.
There are analytic arguments (see \cite{Llibre_Martinez_Simo_1985}) showing that, for a discrete set of values of $\mu$ that are sufficiently small, and for each $h$ sufficiently close to $H(L_1)$, the branches of $W^\un(\lambda_0(h))$ and $W^\st(\lambda_0(h))$ on the side of the heavier primary do not collide with the primary and intersect  transversally along some homoclinic orbit $\gamma_0(h)$, not necessarily unique.
Numerical evidence shows that, for a large range of values of masses $\mu$  and energies  $h$,
the branches of $W^\un(\lambda_0(h))$ and $W^\st(\lambda_0(h))$ on the side of each primary do not collide with either primary and intersect transversally; see, e.g. \cite{koon2000heteroclinic,capinski2012computer}.

Each choice of a transverse homoclinic orbit $\gamma_0(h)$ can be continued in energy $h$ to a family of such homoclinic orbits, which forms a homoclinic manifold $\bigcup_{h\in D} \gamma_0(h)$. Moreover, we can ensure that
$W^\un(\lambda_0(h))$ and $W^\st(\lambda_0(h))$ are contained inside  some compact subset $\mathscr{K}$ of the phase space.
One can always restrict to  a submanifold
\begin{equation}\label{eqn:PCRTBP_homoc}
\Gamma_0\subseteq \bigcup_{h\in D} \gamma_0(h),
\end{equation}
that is a homoclinic channel.

In this way, we can ensure  condition \textbf{(A-ii)}.

The remaining condition is  \textbf{(A-iii)}, which consists of  \textbf{(A-iii-a)} and \textbf{(A-iii-b)}.
These  are explicit non-degeneracy conditions that can separately be verified numerically.

In this way, for the PCRTBP, we can verify the existence of the geometric structures of interest and the corresponding conditions \textbf{(A-i)}, \textbf{(A-ii)}, \textbf{(A-iii)} from  Section \ref{section:unperturbed}.

\textsl{Heteroclinic connections.}
We now focus on the dynamics around the equilibrium points $L_1$ and $L_2$. They satisfy condition  \textbf{(A'-i)}.
For energy levels $h$ with $H(L_1) \lesssim h$, there exists a family $\lambda^1_0(h)$ of Lyapunov orbits  near $L_1$, and  for  $H(L_2) \lesssim h$ there exists  a family $\lambda^2_0(h)$ of Lyapunov orbits  near $L_2$. Moreover, there exist normal form coordinates $(I^1,\theta^1,y^1,x^1)$ and $(I^2,\theta^2,y^2,x^2)$ defined around $L_1$ and $L_2$, respectively,  for some suitable energy ranges. These normal form coordinates are as in Section \ref{sec:coordinates}.

Numerical evidence shows that   families of periodic orbits near $L_1$ and $L_2$ can exist simultaneously, for some energy range.
Therefore, we consider an interval of energies $D$ such that, for $h\in D$ we have the following: there exists families of periodic orbits $\lambda^1_0(h)$ near $L_1$, and   $\lambda^2_0(h)$ near $L_2$, the following sets
\begin{equation}\begin{split}\label{eqn:PCRTBP_NHIM_2}
\Lambda^1_0&=\bigcup_{h\in D} \lambda^1_0(h),\\
\Lambda^2_0&=\bigcup_{h\in D} \lambda^2_0(h),
\end{split} \end{equation}
are NHIM's  for the Hamiltonian flow  of \eqref{eqn:PRCTBP_equations},
and the normal form coordinates $(I^1,\theta^1,y^1,x^1)$ and $(I^2,\theta^2,y^2,x^2)$ are defined in
some neighborhoods of $\Lambda^1_0$ and $\Lambda^2_0$, respectively.

Numerical evidence also shows that
there  exist transverse heteroclinic connections determined by $W^{\un}(\lambda^1_0(h))\cap W^{\st}(\lambda^2_0(h))$ and  $W^{\un}(\lambda^2_0(h))\cap W^{\st}(\lambda^1_0(h))$ for certain ranges of energies.
See, e.g.,  \cite{koon2000heteroclinic,wilczak2003heteroclinic,canalias2006homoclinic,belbruno2010weak}.
In either case, we denote the corresponding family of heteroclinic  orbits by $\gamma_0(h)$.
We consider a  range of energies $h\in D$ for which  this additional condition on the existence of transverse heteroclinic connections is satisfied.

The transverse intersection of the unstable manifold  $W^\un(\Lambda^1_0)$   with the stable manifold $W^\st(\Lambda^2_0)$ define a heteroclinic manifold $\bigcup_{h\in D} \gamma_0(h)$,
which depends on the choice of the family of heteroclinic orbits $\gamma_0(h)$.
One can always restrict to  a submanifold
\begin{equation}\label{eqn:PCRTBP_heteroc_1}
\Gamma_0\subseteq \bigcup_{h\in D} \gamma_0(h),
\end{equation}
that is a heteroclinic channel.

If that is the case, the condition \textbf{(A'-ii)} is verified.

As in the case of homoclinic connections, the remaining condition    \textbf{(A'-iii)} consists of two explicit non-degeneracy conditions that can be verified numerically.

Thus,  one can verify numerically  the conditions \textbf{(A'-i)}, \textbf{(A'-ii)}, \textbf{(A'-iii)}, from Section \ref{sec:heteroclinic}.

It would be of interest to verify if Theorem \ref{thm:main} can be applied when the Lagrange point $L_3$ is also considered;
some numerical results concerning the dynamics around $L_3$ can be found in \cite{barrabes2006invariant,ceccaroni2016halo}. 

In summary, in this section we have outlined how the conditions of the Theorem \ref{thm:main} can be verified in  the PCRTBP.
The theoretical results of Theorem \ref{thm:main} are independent on the application to the PCRTBP.

\section{Coordinate systems and evolution equations}
\label{section:coordinates}
\subsection{New coordinate systems for the unperturbed system}\label{sec:coordinates}

We consider the case of homoclinic connections described by conditions \textbf{(A-i)}, \textbf{(A-ii)}, \textbf{(A-iii)}.
Under these assumptions, the manifolds $W^\un(\Lambda_0)$ and $W^\st(\Lambda_0)$ intersect transversally along the homoclinic channel $\Gamma_0$.

The next proposition states that, in a neighborhood of each $W^\un(\Lambda_0)$ and $W^\st(\Lambda_0)$, there exists a system of symplectic coordinates such that $\Lambda_0$, $W^\un(\Lambda_0)$ and $W^\st(\Lambda_0)$ have very simple descriptions, and moreover the unperturbed Hamiltonian $H_0$ can be written in  a normal form relative to the corresponding coordinates. As before, for $z\in W^{\st,\un}(\Lambda_0)$, we denote $z^\pm=\Omega^{\pm}(z)$.

\begin{prop}[Normal Forms]\label{prop:normal_form} {$ $}
There exist three systems of  symplectic coordinates\footnote{coordinates  obtained from $(p,q)$ via a canonical transformation}, defined for some range of energies $h\in D$,  as follows:
\begin{itemize}
\item [\textbf{(N)}] A coordinate system $(I,\theta,y,x)$ in a neighborhood $\mathscr{N}$ of $\Lambda_0$ such that for $z\in\mathcal{N}$ we have:
\begin{itemize}
\item[\textbf{(N-i)}]  $z\in\Lambda_0$ if and only if $x(z)=y(z)=0$;
\item[\textbf{(N-ii)}]  $z\in W^\un(\Lambda_0)$  if and only if $y=0$, and $z\in W^\st(\Lambda_0)$ if and only if $x=0$;
\item[\textbf{(N-iii)}]  for $z\in W^\un(\Lambda_0)$ we have  $I(z)=I(z^-)$  and $\theta(z)=\theta(z^-)$, and for $z\in W^\st(\Lambda_0)$ we have  $I(z)=I(z^+)$  and $\theta(z)=\theta(z^+)$;
\item[\textbf{(N-iv)}]  $H_0$ restricted to   $\mathscr{N}$ can be written in a normal form
\begin{equation}\label{eqn:normal_n} \begin{split}
    H_0(p,q)=&H_0(I,x y)\\
    =&h_{0}\left(I\right)+(xy) g _{1}\left(I\right)+\left(x y\right)^{2} g_{2}
    (I,x y)
\end{split}\end{equation}
for some analytic functions $h_0=h_0(I)$, $g_1 =g_1 (I)$,  and  $g _2=g_2(I,x y)$.
\end{itemize}
\item[\textbf{(U)}]
A coordinate system $(I^\un,\theta^\un,y^\un,x^\un)$
in a neighborhood $\mathscr{N}^\un$ of $W^\un(\Lambda_0)$ such that for $z\in\mathcal{N}^\un$ we have:
\begin{itemize}
\item[\textbf{(U-i)}]  $z\in\Lambda_0$ if and only if $x^\un(z)=y^\un(z)=0$;
\item[\textbf{(U-ii)}]  $z\in W^\un(\Lambda_0)$ if and only if $y^{\un}=0$;
\item[\textbf{(U-iii)}]  for $z\in W^\un(\Lambda_0)$  we have $I^\un(z)=I^\un(z^-)$, and $\theta^\un(z)=\theta^\un(z^-)$;
\item[\textbf{(U-iv)}]  $H_0$ restricted to   $\mathscr{N}^\un$ can be written in a normal form
\begin{equation}\label{eqn:normal_un} \begin{split}
    H_0(p,q)=&H_0^\un(I^\un,x^\un y^\un)\\
    =&h_{0}\left(I^{u}\right)+(x^{u}y^{u})g _{1}\left(I^{\un}\right)+\left(x^{\un} y^{\un}\right)^{2} g_{2}
    (I^\un,x^\un y^\un)
\end{split}\end{equation}
for some analytic functions $h_0=h_0(I^\un)$, $g_1 =g_1 (I^\un)$,  and  $g _2=g_2(I^\un,x^\un y^\un)$.
\end{itemize}

\item [\textbf{(S)}] A coordinate system
$(I^\st,\theta^\st,y^\st,x^\st)$ in a neighborhood $\mathscr{N}^\st$ of
$W^\st(\Lambda_0)$ such that for $z\in\mathcal{N}^\st$ we have:
\begin{itemize}
\item[\textbf{(S-i)}]  $z\in\Lambda_0$ if and only if $x^\st(z)=y^\st(z)=0$;
\item[\textbf{(S-ii)}]  $z\in W^\st(\Lambda_0)$ if and only if $x^{\st}=0$;
\item[\textbf{(S-iii)}]  for $z\in W^\st(\Lambda_0)$  we have  $I^\st(z)=I^\st(z^+)$, and $\theta^\st(z)=\theta^\st(z^+)$;
\item[\textbf{(S-iv)}] $H_0$ restricted to   $\mathscr{N}^\st$ can be written in a normal form
\begin{equation}\label{eqn:normal_st} \begin{split}
    H_0(p,q)=&H_0^\st(I^\st,x^\st y^\st)\\=&h_{0}\left(I^{\st}\right)+(x^{\st}y^{\st})g _{1}\left(I^{\st}\right)+\left(x^{\st} y^{\st}\right)^{2}g _{2}
    (I^\st,x^\st y^\st)
\end{split}\end{equation}
for some analytic functions $h_0=h_0(I^\st)$,  $g_1=g_1(I^\st)$,  and  $g_2=g_2(I^\st,x^\st y^\st)$.
\end{itemize}
\end{itemize}
The coordinate systems $(I^\un,\theta^\un,y^\un,x^\un)$ and $(I^\st,\theta^\st,y^\st,x^\st)$ coincide with $(I,\theta,y,x)$ on  $\mathscr{N}$, i.e.,
    \begin{equation}\label{eqn:US}
      (I^\un,\theta^\un,y^\un,x^\un)=(I^\st,\theta^\st,y^\st,x^\st)= (I,\theta,y,x) \textrm { on } \mathscr{N}.
    \end{equation}

The function $h_0,g_1,g_2$ that appear in \eqref{eqn:normal_n}, \eqref{eqn:normal_un} and
\eqref{eqn:normal_st} are the same.
\end{prop}

\begin{proof}

Part \textbf{(N)} follows  from \cite{giorgilli2001unstable}, so we will not give  a detailed  proof.

We only summarize  the procedure  to obtain the normal form in a neighborhood of a center-saddle equilibrium point for a $2$-degrees of freedom Hamiltonian $H_0$.
First, the Hamiltonian $H_0$ is expanded in a Taylor series around that equilibrium point (shifted to the origin) as
\[H_0(p,q)=H_2(p,q)+H_3(p,q)+H_4(p,q)+\ldots,\]
where $H_j(p,q)$ is an homogeneous polynomial of degree $j$ in the variables $(p_1,p_2,q_1,q_2)$.
Then, by making  a linear  canonical change of coordinates $(p,q)\mapsto (x,y)$, with the eigenvectors of the linearized system given by
$J\nabla H_2(0,0)$ as the axes of the new system, the quadratic part $H_2$ of $H_0$ can be written in   the new coordinates $(x,y)\in\R^4$  as
\[H_2(x,y)=\lambda x_1y_1+\frac{\omega}{2}(x_2^2+y_2^2),\]
where $\pm\lambda_1=\pm \lambda$ and $\pm\lambda_2=\pm i\omega $ are the eigenvalues of $J\nabla H_2(0,0)$.
Then, via another linear canonical  change of coordinates
\[x_1=\xi_1,\quad y_1=\eta_1,\quad x_2=\frac{\xi_2+i \eta_2 }{\sqrt{2}},\quad \quad y_2=\frac{i\xi_2 +\eta_2}{\sqrt{2}},\] we obtain $H_2$ written in complex variables as
\[H^N_2(\xi,\eta)=\lambda \xi_1\eta_1+i \omega  \xi_2\eta_2=\lambda_1 \xi_1\eta_1+\lambda_2  \xi_2\eta_2.\]
The next step is  to apply a sequence of changes of coordinates to kill all monomials for which the exponent
of $\xi_j$ is different from the exponent of $\eta_j$. Since the  eigenvalue  $\lambda_1=\lambda$ is real and $\lambda_2=i\omega$ is imaginary, there are no small divisors.
The only possible source of divergence is due to the use of Cauchy's estimates
for the derivatives required by the normalization procedure. In \cite{giorgilli2001unstable}
the accumulation of derivatives is controlled via a  KAM technique.
The process can be continued to any order.

Thus, in the limit  $H_0$ can be written as an  expansion
\[H^N_0(\xi,\eta)=H^N_2(\xi_1\eta_1,\xi_2\eta_2)+H^N_3(\xi_1\eta_1,\xi_2\eta_2)+H^N_4(\xi_1\eta_1,\xi_2\eta_2)+\ldots,\]
where $H_j$ is an homogeneous polynomial in $\xi_1\eta_1,\xi_2\eta_2$ of degree $j$.
The series expansion of $H_0^N$ is convergent in a neighborhood $\bar{\mathscr{N}}$ of the origin,
and the coordinate change $x=x(\xi,\eta)$, $y=y(\xi, \eta)$
is canonical and given in terms of convergent series.

There exist periodic orbits $\lambda_0(h)$ around the equilibrium point for all energy levels $h$ sufficiently close
to the energy level of the equilibrium point. This implies that the NHIM $\bar{\Lambda}_0=\bigcup_{h\in\bar{D}}\lambda_0(h)$
is contained in $\bar{\mathscr{N}}$, for some suitable energy range  $h\in\bar{D}$.

To express $H_0^N$ in action-angle coordinates one  applies the canonical transformation
\[\xi_2=\sqrt{\bI}\exp( i\btheta),\quad \eta_2=-i\sqrt{\bI}\exp(-i\btheta).\]
Finally, denote  $\bx=\xi_1$, $\by=\eta_1$.
We obtain the normal form
\begin{equation}\label{eqn:normal1}
H_0^N(\bI,\bx\by)=\lambda \bx\by+{\omega}\bI+H_3(\bI,\bx\by)+H_4(\bI,\bx\by)+\ldots.\end{equation}
Moreover, in these coordinates the following hold:
\begin{itemize}
\item[(i)] The normally hyperbolic invariant manifold $\bar{\Lambda}_0$ is given by $\bx=\by=0$, and each periodic orbit in $\bar{\Lambda}_0$ corresponds to a level set of $\bI$;
\item[(ii)] The local unstable invariant manifold $W^\un_{\bar{\mathscr{N}}}(\bar{\Lambda}_0)$ is given by $\by=0$;
\item[(iii)] The local stable invariant manifold $W^\st_{\bar{\mathscr{N}}}(\bar{\Lambda}_0)$ is given by $\bx=0$.
\end{itemize}

The equations of motion are
\begin{equation}\begin{split}
\frac{d}{dt} \bI=&0,\\
\frac{d}{dt} \btheta=&\omega +\frac{\partial H_3^N}{\partial \bI}+\frac{\partial H_4^N}{\partial \bI}+\ldots,\\
\frac{d}{dt} \by=&-\lambda \by-\frac{\partial H_3^N}{\partial \bx}-\frac{\partial H_4^N}{\partial \bx}+\ldots,\\
\frac{d}{dt} \bx=&\lambda \bx+\frac{\partial H_3^N}{\partial \by}+\frac{\partial H_4^N}{\partial \by}+\ldots.
\end{split}\end{equation}

Note that Hamiltonian $H_0^N$ on  $\bar{\mathscr{N}}$ has two first integrals $\bI$ and $\bx\by$, which are independent and in involution.

This implies that, if $z\in W^\un_{\bar{\mathscr{N}}}(\bar{\Lambda}_0)$ (resp. $z\in W^\st_{\bar{\mathscr{N}}}(\bar{\Lambda}_0)$), since $\bx\by=0$, we have $\bI(z)=\bI(z^-)$ (resp. $\bI(z)=\bI(z^+)$).

The Hamiltonian $H_0^N$  restricted to $\bar{\Lambda}_0$ is given by
\[h_0(\bI):={\omega}\bI+H_3^N(\bI)+H_4^N(\bI)+\ldots,\]
where we denote $H^N_0(\bI,0)=h_0(\bI)$, and $H^N_j(\bI,0)=H^N_j(\bI)$ for all $j$.
Each Lyapunov orbit $\lambda_h$ in $\bar{\Lambda}_0$ corresponds to a unique level set of  $\bI$, so we can write $\lambda_0(h)=\lambda_0 (\bI)$.

By truncating the expansion \eqref{eqn:normal1} at the first order we can write
\begin{equation}\label{eqn:normal2}
H_0^N(\bI,\btheta,\by,\bx)
=h_0(\bI)+(\bx\by)g_1(\bI)+(\bx\by)^2g_2(\bI,\bx\by),
\end{equation}
for some analytic functions $h_0=h_0(\bI)$, $g_1=g_1(\bI,\bx\by)$, and $g_2=g_2(\bI,\bx\by)$.

For points $z\in W^\un_{\bar{\mathscr{N}}}(\bar{\Lambda}_0)$ or $z\in W^\st_{\bar{\mathscr{N}}}(\bar{\Lambda}_0)$, since $\bx\by=0$, we have
\[\frac{d}{dt}\theta(z)=\frac{\partial h_0}{\partial \bI}(\bI(z)).\]
This implies  that if $z\in W^\un_{\bar{\mathscr{N}}}(\bar{\Lambda}_0)$ (resp. $z\in W^\st_{\bar{\mathscr{N}}}(\bar{\Lambda}_0)$),
we have $\btheta(z)=\btheta(z^-)$ (resp. $\btheta(z)=\btheta(z^+)$).

The coordinate system $(\bI,\btheta, \by,\bx)$ constructed above is not yet the coordinate system from part \textbf{(N)}. The desired coordinate system $(I,\theta,y,x)$ will be constructed below by flowing in time the coordinate system $(\bI,\btheta, \by,\bx)$.

Now we construct the coordinate system claimed in part \textbf{(U)}.

We extend the coordinate system  $(\bI,\btheta,\by,\bx)$ along the flow to a neighborhood $\mathscr{N}^\un$ of $W^\un(\bar{\Lambda}_0)$, up to a neighborhood of the homoclinic manifold, as follows.
Let $T>0$ be a time such that $\Phi^T_0(\bar{\mathscr{N}})\supseteq \Gamma$.
Let  $\mathscr{N}^\un:= \Phi^T_0(\bar{\mathscr{N}})$.
Each point $z\in \mathscr{N}^\un$ is of the form $z=\Phi^T_0(\zeta)$ with $\zeta\in\bar{\mathscr{N}}$.
We define the coordinates $(I^\un,\theta^\un, y^\un, x^\un)$  of $z$ to be equal to the coordinates $(\bI,\btheta,\by,\bx)$ of $\zeta$, or equivalently
\begin{equation}\label{eqn:coord_u}
  (I^\un,\theta^\un,y^\un, x^\un)(z)=(\bI,\btheta,\by,\bx)(\Phi^{-T}_0(z)).
\end{equation}

The restriction of the coordinates $(I^\un,\theta^\un,y^\un, x^\un)$ to $\bar{\mathscr{N}}\cap\mathscr{N}^\un$ is given by $(\bI,\btheta,,\by\bx)\circ \Phi^{-T}_0$.

Since the coordinates  $(I^\un,\theta^\un,y^\un, x^\un)$ of a point $z$ are the coordinates $(\bI,\btheta,\by,\bx)$ of $\Phi^{-T}_0(z)$, then they are symplectic, and they yield the same normal form expansion of $H_0$ as \eqref{eqn:normal2}.
More precisely,
\begin{equation*}
\begin{split}
H(I^\un ,\theta^\un,& y^\un,x^\un)\\
=&H(\bI\circ\Phi^{-T}_0 ,\btheta\circ\Phi^{-T}_0,\by\circ\Phi^{-T}_0 ,\bx\circ\Phi^{-T}_0)\\
=&h_0(\bI\circ \Phi^{-T}_0)\\
&+\left ((\bx\circ\Phi^{-T}_0)\cdot (\by\circ\Phi^{-T}_0)\right)g_1 (\bI\circ\Phi^{-T}_0)\\
&+\left ((\bx \circ \Phi^{-T}_0)\cdot (\by\circ \Phi^{-T}_0)\right)^2g_2 \left((\bI\circ \Phi^{-T}_0,\bx\circ \Phi^{-T}_0) \cdot (\by \circ\Phi^{-T}_0)\right)\\
=&h_0(I^\un)+(x^\un y^\un)g_1 (I)+(x^\un y^\un)^2g_2 (I^\un,x^\un y ^\un).
\end{split}
\end{equation*}

Now we construct the coordinate system claimed in part \textbf{(S)}.

We extend the coordinate system  $(I,\theta,y,x)$ along the flow to a neighborhood $\mathscr{N}^\st$ of $W^\st(\bar{\Lambda}_0)$, up to the homoclinic manifold, as follows.
Start with the coordinates $(\bI,\btheta,\by,\bx)\circ \Phi^{-2T}_0$ defined in the neighborhood $ \Phi^{-2T}_0(\bar{\mathscr{N}})$ of $\bar{\Lambda}_0$.
Let  $\mathscr{N}^\st:= \Phi^{-3T}_0(\bar{\mathscr{N}})$.
Each point $z\in \mathscr{N}^\st$ is of the form $z=\Phi^{-T}_0(\zeta)$ with $\zeta\in\Phi^{-2T}_0(\bar{\mathscr{N}})$.
We define the coordinates $(I^\st,\theta^\st,y^\st, x^\st)$  of $z$ to be equal to the coordinates $(\bI,\btheta,\by,\bx)\circ \Phi^{-2T}_0$ of $\zeta$, or equivalently
\begin{equation}\label{eqn:coord_s}
(I^\st,\theta^\st,y^\st,x^\st)(z)=(\bI,\btheta,\by,\bx)(\Phi^{-2T}_0(\Phi^T_0(z)))=(\bI,\btheta,\by,\bx)
(\Phi^{-T}_0(z)).
\end{equation}

Since the coordinates  $(I^\st,\theta^\st,y^\st, x^\st)$ of a point $z$ are the coordinates $(\bI,\btheta,\by,\bx)\circ \Phi^{-2T}_0$ of $\Phi^{T}_0(z)$, then they are symplectic, and they yield the same normal for expansion of $H_0$ as in \eqref{eqn:normal2}, that is
\[H^N_{\mathscr{N}^\st}(I^\st,\theta^\st, y^\st, x^\st)=h_0(I^\st)+(x^\st y^\st)g_1(I)+(x^\st y^\st)^2g_2(I^\st,x^\st y ^\st).\]

The restriction of the coordinates $(I^\st,\theta^\st,y^\st, x^\st)$ to $\bar{\mathscr{N}}\cap\mathscr{N}^\st$ is given by $(\bI,\btheta,\by,\bx)\circ \Phi^{-T}_0$.

Now we define the coordinate system claimed in  part \textbf{(N)}.

We construct the coordinate system $(I,\theta,y, x)$ in a neighborhood $\mathscr{N}:=\bar{\mathscr{N}}\cap
 \mathscr{N}^\un\cap  \mathscr{N}^\st$ by
\begin{equation}\label{eqn:coord_n}
 (I,\theta,y,x)(z)=(\bI,\btheta,\by,\bx)(\Phi^{-T}_0(z)).
\end{equation}
In terms of the coordinates $(I,\theta,y,x)$, the Hamiltonian has the same normal form expansion as in \eqref{eqn:normal2}, that is
\[H^N_{\mathscr{N}}(I,\theta,y,x)=h_0(I)+(x y)g_1(I)+(x y)^2g_2(I,x y ).\]

We note that $\mathscr{N}$ is a small neighborhood of the equilibrium point in the phase space. Thus we restrict the Hamiltonian  to an energy range $h\in D$ such that $\Lambda_0=\bigcup_{h\in D}\lambda_h\subseteq \mathscr{N}$.
Moreover, we choose $D$ such that $W^\un(\Lambda_0)$, $W^\st(\Lambda_0)$ are contained in $\mathscr{K}$, where the set
$\mathscr{K}$ is as in condition \textbf{(A-ii)}.

By the above constructions, the coordinates $(I,\theta,y,x)$, $(I^\un,\theta^\un,y ^\un, x^\un)$ and $(I^\st,\theta^\st, y^\st, x^\st)$ satisfy the properties listed in Proposition \ref{prop:normal_form}.
\end{proof}

We now formulate some  non-degeneracy condition, in terms of the above normal form coordinates,  that $H_0$ must satisfy for our results.
\begin{itemize}
\item[\textbf{(A-iii-a)}] \emph{The Hamiltonian $H_0$, written in the normal form  given by Proposition \ref{prop:normal_form}
satisfies:}
\begin{equation}
\label{eqn:nondegeneracy1}
\begin{split}
\frac{\partial h_0}{\partial I}(I_0)\neq & 0 \textrm { for all } I_0=I_0(h) \textrm { with } h\in D,\\
g_1(I_0)\neq & 0  \textrm { for all } I_0=I_0(h) \textrm { with } h\in D.
\end{split}
\end{equation}
\end{itemize}

\begin{rem}\label{rem:coord_mismatch}
It is important to note that the coordinates $(I^\un,\theta^\un,x^\un,y^\un)$  and
$(I^\st,\theta^\st,x^\st,y^\st)$ do not generally agree at  homoclinic points  away from the Lyapunov orbit, where both coordinate systems are well defined.
Nevertheless,   for any homoclinic point $z\in M_h\cap  W^\un(\Lambda_0)\cap W^\st(\Lambda_0)$, we have that $I^\un(z)=I^\st(z)=I_h$.
\end{rem}

\begin{rem}
The above result on the existence of a convergent normal form in a neighborhood of a center-saddle point,  obtained via a convergent  canonical coordinate transformation,  is  valid for 2-degrees of  freedom Hamiltonian systems.
For higher degree of freedom  Hamiltonian systems,  the same result is true under some additional non-resonance conditions (see \cite{giorgilli2001unstable}).
A related approach to the normal form that we use here can be found in \cite{moser1958generalization}.
A numerical methodology for the effective computations of  normal forms is developed in \cite{jorba1999methodology}.
\end{rem}

\begin{rem}\label{rem:normal_form}
Now we discuss the case when we have heteroclinic connections between two  NHIMs $\Lambda^1_0$  around $L^1$ and $\Lambda^2_0$ around $L^2$,
as in Section \ref{sec:heteroclinic}.
The manifolds $W^\un(\Lambda^1_0)$ and $W^\st(\Lambda^2_0)$ are assumed to intersect transversally.

The construction of the normal form coordinates from Proposition \ref{prop:normal_form} only works in a small neighborhood of the equilibrium point. In the case of heteroclinic connections, since $L^1$ and $L^2$ are on different energy level,
the theory does not guarantee the simultaneous  existence two normal form coordinate system  around $L^1$ and $L^2$ respectively, for some common energy range.

In this case, we need to make a separate assumption that there exist  two systems of normal form coordinates
around $L^1$ and $L^2$, for some common energy range.
Indeed, this assumption is already made in \textbf{(A'-i-b)}.

Based on this assumption, we can construct, as  in the proof of Proposition \ref{prop:normal_form}, two systems of coordinates
\begin{itemize}
\item $(I^{1,\un},\theta^{1,\un},y^{1,\un},x^{1,\un})$  in a neighborhood $\mathscr{N}^{1,\un}\subseteq\mathscr{K}$
of $W^\un(\Lambda^1_0)$,
\item $(I^{2,\st},\theta^{2,\st},y^{2,\st},x^{2,\st})$  in a neighborhood $\mathscr{N}^{2,\st} \subseteq\mathscr{K}$
of $W^\st(\Lambda^2_0)$,
\end{itemize}
so that they satisfy properties similar to those in the case of homoclinic connections.

\end{rem}

\subsection{The scattering map for the unperturbed  system}
\label{sec:scattering_PCRTBP}
Consider the scattering map ${\sigma}_0$ associated to ${\Gamma}_0$.
We will express the scattering map in terms of the coordinates
$(I^\un,\theta^\un,y^\un,x^\un)$  and
$(I^\st,\theta^\st,y^\st,x^\st)$.

Consider a homoclinic point $z_0\in W^\un (\lambda_{h})\cap W^\st (\lambda_{h})$. Both coordinate systems  $(I^\un,\theta^\un,y^\un,x^\un)$  and $(I^\st,\theta^\st,y^\st,x^\st)$ are defined in the neighborhood of $z_0$.

By Proposition \ref{prop:normal_form}  the action coordinate of $z_0\in{\Gamma}_0$ is the same as the action of the unstable and stable foot-points  $z_0^-,z_0^+\in{\Lambda}_0$,  that is $I(z_0^-)=I^\un(z_0^-)=I^\un(z_0)= I^\st(z_0) = I^\st(z_0^+)=I(z_0^+)$.
Therefore the scattering map ${\sigma}_0$ preserves the $I$-coordinate.
Hence ${\sigma}_0$ is a phase shift on  $I$-level-sets in ${\Lambda}_0$ wherever it is defined:
\begin{equation}\label{eqn:sigma0ii}{\sigma}_0(I,\theta)=(I,\theta+\Delta(I)).\end{equation}

In general, $\theta^\un(z_0)\neq \theta^\st(z_0)$.
It is easy to see that the phase-shift determined by the unperturbed scattering map is
given by the `mismatch' between the two angle coordinates.

\begin{prop}\label{prop:mismatch} Let $h$ be a fixed energy level and let $z_0\in\Gamma_0\cap M_{h}$.

(i) The angle mismatch $\theta^\st(z_0)-\theta^\un(z_0)$ is a constant that depends only on $h$, so we write it as $\theta^\st(h)-\theta^\un(h)$.

(ii) The scattering map $\sigma_0$  is given by $(I,\theta)\mapsto \sigma_0(I,\theta)=(I,\theta+\Delta(I))$, where $\Delta(I)=\theta^\st(h)-\theta^\un(h)$ for $I=I_{h}$.
\end{prop}
\begin{proof}

(i) If $z_0$ is a point in $\Gamma_0\cap M_h$, then $\Gamma_0\cap M_{h}$ consists of  points of the form $\Phi^t_0(z)$, $t\in [t_1,t_2]$, for some interval $[t_1,t_2]$ containing $0$.

From $\Omega^+(z_0)=z_0^+ $ and
$\Omega^-(z_0)=z_0^-$, by the equivariance property \eqref{eqn:equivariance}
we have $\Omega^+(\Phi^t_0 (z_0))=\Phi^t_0 (z_0^+)$ and
$\Omega^-(\Phi^t_0 (z_0))=\Phi^t_0 (z_0^-)$ for $t\in [t_1,t_2]$.
By By Proposition \ref{prop:normal_form} \textbf{(U-iii)} and \textbf{(S-iii)},
and by the fact that the dynamics restricted to
$\lambda_{h}$ is a rigid rotation in $\theta$,  we have
\begin{equation*}
\begin{split}
\theta^\st(\Phi^t_0 (z_0))-\theta^\un(\Phi^t_0 (z_0))=& \theta(\Phi^t_0 (z_0^+))-\theta(\Phi^t_0(z_0^-))\\
=& \theta(z_0^+)-\theta(z_0^-),
\end{split}\end{equation*}
which depends only on the energy level $h$.

(ii) In action-angle coordinates, if $z_0^-=(I,\theta)$, then $\sigma_0(z_0^-)=z_0^+=(I,\theta+\Delta(I))$.
From (i) we immediately deduce that  $\Delta(I)=\Delta(I(h))= \theta^\st(h)-\theta^\un(h)$.
\end{proof}

Turning now the attention to the unperturbed  homoclinic channel $\tGamma_0$ in the extended phase space,
this can be parametrized in terms of the coordinates $(I^\un,\theta^\un,t)$, as well as in terms of the  coordinates $(I^\st,\theta^\st,t)$:
\begin{equation}
\begin{split}
   \tGamma_0=& \{(I^\un,\theta^\un,y^\un, x^\un,t)\,|\, y^\un=0,\,x^\un=x^\un_0(I^\un,\theta^\un,t)\}, \\
   =& \{(I^\st,\theta^\st, x^\st,y^\st, t)\,|\, y^\st_0=y^\st(I^\st,\theta^\st,t),\,x^\st=0\}.
\end{split}
\end{equation}
That is, $\tGamma_0$ is a graph over the $(I^\un,\theta^\un,  t)$-variables, as well as a graph over
the $(I^\st,\theta^\st,  t)$-variables.
Each homoclinic point $\tz_0\in\tilde{\Gamma}_0$ is associated to unique $I^\un=I^\st=I_0$, $\theta^\un$, $\theta^\st$ with $\theta^\st(\tz_0)-\theta^\un(\tz_0)=\Delta (I_0)$.
Moreover, for the points on the homoclinic orbit $\tPhi^\tau(\tz_0)$, we also have that
$\theta^\st(\tPhi^\tau(\tz_0))-\theta^\un(\tPhi^\tau(\tz_0))=\Delta (I_0)$.

The corresponding scattering map is given by
\[\tilde{\sigma}_0(I,\theta,t)=(I,\theta+\Delta(I),t).\]

Define \begin{equation}\begin{split}
\mathscr{D}_{I_0}=&\{\tz_0\,|\,\theta^\st(\tz_0)-\theta^\un(\tz_0) = \Delta (I_0),\, I^\st(\tz_0)-I^\un(\tz_0) = 0\}.
\end{split}\end{equation}

We require the following condition
\begin{itemize}
\item[\textbf{(A-iii-b)}] \emph{For all $I_0=I_0(h)$ with $h\in D$, $(\Delta(I_0),0)$ is a regular value of the function \[\tz_0\mapsto (\theta^\st(\tz_0)-\theta^\un(\tz_0), I^\st(\tz_0)-I^\un(\tz_0)).\]}
\end{itemize}

Condition \textbf{(A-iii-b)} implies that $\mathscr{D}$ is  a codimension-$2$ surface in $\tM$, which intersects transversally $\tGamma_0$ along a $1$-dimensional curve.

At this point, we have completed the formulation of condition \textbf{(A-iii)} in Theorem \ref{thm:main}, consisting of
\textbf{(A-iii-a)}, given in Section \ref{sec:coordinates}
and \textbf{(A-iii-b)}, given above.

\subsection{Perturbed evolution equations}\label{sec:evolution}

In the sequel, we will identify the vector fields $\X^0$ and $\X^1$ with derivative operators acting on functions.
In general, given a smooth vector field $\X$ and a smooth function $f$ on a manifold $M$, and  $(z_j)_{j\in\{1,\ldots,\dim(M)\}}$   a system of local coordinates, then
\begin{equation}
\label{eqn:vf_derivative}(\X  f)(z) = \sum_j (\X)_j (z) (\partial_{z_j} f)(z).
\end{equation}

Consider one of the coordinate systems defined in Section~\ref{sec:coordinates}.
To simplify notation, we will denote such a  coordinate system by $(I,\theta,x,y)$. Below we provide evolution equations of these coordinates, expressing the time-derivative of each coordinate along a solution of the perturbed system.
We include the expression for a general perturbation, as well as for the   case when the perturbation is Hamiltonian:
\begin{equation}\label{eqn:evolution3}
\begin{split}
\frac{d}{dt} I &= (\X^0+\eps \X^1)(I) =  -\frac{\partial{H_0}}{\partial \theta} + \eps\X^1(I) \\&=
 -\frac{\partial{H_0}}{\partial \theta}-\eps \frac{\partial{H_1}}{\partial \theta}.
\end{split}
\end{equation}
\begin{equation}
\label{eqn:evolution4}
\begin{split}\frac{d}{dt} \theta &=  (\X^0+\eps \X^1)(\theta) =  \frac{\partial{H_0}}{\partial I} + \eps\X^1(\theta) \\&=
 \frac{\partial{H_0}}{\partial I}+\eps \frac{\partial{H_1}}{\partial I}.
\end{split}
\end{equation}
\begin{equation}\label{eqn:evolution2}
\begin{split}
\frac{d}{dt} y &=   \X^0(y) + \eps\X^1(y)  = -\frac{\partial{H_0}}{\partial x} + \eps\X^1(y)\\&=
-\frac{\partial{H_0}}{\partial x}-\eps\frac{\partial{H_1}}{\partial x}.
\end{split}
\end{equation}
\begin{equation}\label{eqn:evolution1}
\begin{split}
\frac{d}{dt} x &=   \X^0(x) + \eps\X^1(x)  =\frac{\partial{H_0}}{\partial y} + \eps\X^1(x)\\&=
\frac{\partial{H_0}}{\partial y}+\eps\frac{\partial{H_1}}{\partial y}.
\end{split}
\end{equation}

\section{Proof of the main result}
\label{section:proofs}
In this section we prove Theorem \ref{thm:main}.
\subsection{Perturbed normally hyperbolic invariant manifolds}\label{sec:perturbed_NHIM}
In this section we prove the assertions (i) and (ii) of Theorem \ref{thm:main}.

We only give the details in the case when  $H_0$ satisfies  the conditions \textbf{(A-i)}, \textbf{(A-ii)}, \textbf{(A-iii)}. The case when $H_0$ satisfies \textbf{(A'-i)}, \textbf{(A'-ii)}, \textbf{(A'-iii)} follows similarly.

\subsubsection{Persistence of the normally hyperbolic invariant manifold under perturbation}
We have that $\Lambda_0$ is a NHIM for the flow $\Phi^t_0$ of $\X^0$.
Then $D\Phi^t_0(z)$ satisfies expansion/contraction rates as in  Appendix \ref{sec:NHIM}, for all $z\in\Lambda_0$, where we denote the  constant and the expansion and contraction rates by $C$,  $\lambda_-$, $\lambda_+$, $\lambda_c$, $\mu_c$, $\mu_-$, $\mu_+$, respectively.

It is immediate that $\tLambda_0=\Lambda_0\times \R$ is a NHIM for the flow $\tPhi^\tau_0$ of the extended system \eqref{eqn:generalperturbation_t}.

Under the assumptions of Theorem \ref{thm:main}, $\X^1=\X^1(z,t;\eps)$ is
uniformly differentiable in all variables.  The theory of normally hyperbolic invariant manifolds,
\cite{Fenichel71,HirschPS77,Pesin04} (a handy summary of
the results of the theory is \cite{DelshamsLS06a}),  asserts that there exists $\eps_1$ such that the manifold $\tLambda_0$
persists as a  normally hyperbolic manifold $\tLambda_\eps$, for all $|\eps| < \eps_1$, which is locally invariant under the flow $\tPhi^\tau_\eps$.
The persistent NHIM $\tLambda_\eps$ is $O(\eps)$ close in the $\mathcal{C}^\ell$-topology to $\tLambda_0$, where $\ell$ is as in \ref{eqn:ratesdifferentiable}.  The locally invariant manifolds are in fact invariant manifolds for an extended system, and they depend on the extension. Hence,  they do not need to be unique.
Nevertheless, we point out that, given a family of systems, it is possible
to choose the invariant manifolds in such a way that the invariant
manifods depend
smoothly on parameters, as well as the stable and unstable bundles and
the stable and unstable manifolds.

For the  perturbed NHIM $\tLambda_\eps$, $|\eps| < \eps_1$, there exists an invariant splitting of the tangent bundle $T\tLambda_\eps$, similar  to that in \eqref{eqn:NHIM_splitting}, so that $D\tPhi^\tau_\eps(\tz)$ satisfies expansion/contraction relations similar to those in
\eqref{eqn:NHIM_rates} for all $\tz\in\tilde{\Lambda}_\eps$, for some constants $\tilde C$, $\tilde\lambda_-$, $\tilde\lambda_+$, $\tilde\lambda_c$, $\tilde\mu_c$, $\tilde\mu_-$, $\tilde\mu_+$.
These constants are independent of $\eps$, and can be chosen as close as desired to the unperturbed ones, that is, to
$C$, $\lambda_-$, $\lambda_+$, $\lambda_c$, $\mu_c$, $\mu_-$, $\mu_+$, respectively,  by choosing $\eps_1$ suitably small.

There exist unstable and stable manifolds $W^\un(\tLambda_\eps)$, $W^\st(\tLambda_\eps)$ associated to $\tLambda_\eps$,
and there exist corresponding projection maps $\Omega^-:W^\un(\tLambda_\eps)\to\tLambda_\eps$, and $\Omega^+: W^\st(\tLambda_\eps)\to\tLambda_\eps$.

For $\tilde{z}^+=\Omega^+(\tilde{z})$, with $\tilde{z}\in W^\st(\tLambda_\eps)$ we have
\begin{equation}\label{eqn:convergence_s}
d( \tilde{\Phi}^\tau_\eps(\tilde{z}), \tilde{\Phi}^\tau_\eps(\tilde{z}^+) ) \le
 C_{\tilde{z}} e^{\tau \tilde{\lambda}_+}, \quad \textrm { for all } \tau \ge
0,
\end{equation}
and for $\tilde{z}^-=\Omega^-(\tilde{z})$, with $\tilde{z}\in W^\un(\tLambda_\eps)$ we have
\begin{equation}\label{eqn:convergence_u}
d(  \tilde{\Phi}^\tau_\eps(\tilde{z}), \tilde{\Phi}^\tau_\eps(\tilde{z}^-) \le
C_{\tilde{z}} e^{\tau \tilde{\mu}_-}, \quad \textrm { for all } \tau \leq 0,
\end{equation}
for some ${C}_{\tilde z}>0$. The constant $\tilde{C}_{\tilde{z}}$ can be chosen uniformly bounded, provided that we restrict to  $z$ to the compact neighborhood $\mathscr{K}$ given by \textbf{(A-ii)}, and we use the fact that $\X^1=\X^1(z,t;\eps)$ is uniformly differentiable in all variables.

To simplify notation, from now on we will drop the symbol $\tilde{}$ from $\tilde C$, $\tilde{C}_{\tilde z}$ $\tilde\lambda_-$, $\tilde\lambda_+$,  $\tilde\mu_-$, $\tilde\mu_+$, $\tilde\lambda_c$, $\tilde\mu_c$.

In the sequel, we will fix  a choice $\tLambda_\eps$, and all computations will be performed relative to that choice.
Nevertheless the estimate for the perturbed scattering map $\tilde{\sigma}_\eps$ are independent of the choice of the locally invariant manifold  $\tLambda_\eps$.

\subsubsection{Persistence of the transverse intersection of the hyperbolic invariant manifolds under perturbation}\label{sec:peristence_transverse}
For the unperturbed system the unstable and stable manifolds $W^\un(\tLambda_0)$, $W^\st(\tLambda_0)$ intersect transversally along the $3$-dimensional homoclinic channel $\tilde{\Gamma}_0$.
By the persistence of transversality under small perturbations, it follows that
$W^\un(\tLambda_\eps)$, $W^\st(\tLambda_\eps)$  intersect transversally along $\tilde{\Gamma}_\eps$, for all $|\eps| < \eps_1$, provided $\eps_1$ is chosen small enough. The condition  \eqref{goodtransversal}  in the definition of a homoclinic/heteroclinic channel is also a transversality-type condition, so it is also persistent under small perturbations.
We conclude that (i) and (ii) from Theorem \ref{thm:main} hold true for all $|\eps| < \eps_1$, provided $\eps_1$ is chosen small enough.

Recall from Section \ref{sec:scattering_PCRTBP} that the unperturbed homoclinic channel $\tGamma_0$ can be described as a graph over the $(I^\un,\theta^\un,  t)$-variables, as well as a graph over
the $(I^\st,\theta^\st,  t)$-variables.
Therefore, the  perturbed homoclinic channel $\tGamma_\eps$, for $|\eps| < \eps_1$,
 can also be described as a graph over the $(I^\un,\theta^\un,  t)$-variables, as well as a graph over
the $(I^\st,\theta^\st,  t)$-variables.
Therefore, each homoclinic point $\tz_\eps\in\tilde{\Gamma}_\eps$ is associated to unique coordinate triples $(I^\un,\theta^\un,t)$, and $(I^\st,\theta^\st,t)$.
More precisely, we have
\begin{equation}
\begin{split}
   \tGamma_\eps =& \{(I^\un,\theta^\un,y^\un, x^\un, t)\,|\, y^\un=y^\un_\eps(I^\un,\theta^\un,t),\,x^\un=x^\un_\eps(I^\un,\theta^\un,t) \}, \\
   \tGamma_\eps =& \{(I^\st,\theta^\st,y^\st,  x^\st,t)\,|
   \,y^\st=y^\st_\eps(I^\st,\theta^\st,t),\,x^\st=x^\st_\eps(I^\st,\theta^\st,t)\},
\end{split}
\end{equation}
with $y^\un_\eps(I^\un,\theta^\un,t)=O(\eps)$ and $x^\st_\eps(I^\st,\theta^\st,t)=O(\eps)$.

By condition \textbf{(A-iii-b)}, since $\mathscr{D}_{I_0}$ intersects $\tGamma_0$ transversally,
it follows that $\mathscr{D}_{I_0}$ intersects $\tGamma_\eps$ for $\eps$ sufficiently small.
Therefore, given $\tz_0$ in $\tGamma_0$ we will associate to it  a homoclinic point $\tz_\eps \in \mathscr{D}_{I_0}$, that is,  satisfying the conditions
\begin{equation}\label{eqn:theta_st_minus_theta_un}
\begin{split}
\theta^\st(\tz_\eps)-\theta^\un(\tz_\eps)=&\theta^\st(\tz_0)-\theta^\un(\tz_0),\\
I^\st(\tz_\eps)-I^\un(\tz_\eps)=&0.
\end{split}
\end{equation}

Note that a homoclinic point $\tz_\eps$ satisfying these conditions is not uniquely defined.
We can impose on $\tz_\eps$  an additional condition, for instance $t(\tz_\eps)=t(\tz_0)$, or
$I^\st(\tz_\eps)=I^\un(\tz_\eps)=I_0$. Such an extra condition can be useful for applications.

In the sequel we will compare the scattering map associated to $\tz_0$ with the scattering map associated to $\tz_\eps$ satisfying \eqref{eqn:theta_st_minus_theta_un}.

\subsection{Perturbed scattering map}
\label{sec:perturbed_scattering}
In this section we prove the assertion (iii) of Theorem \ref{thm:main}.

We start with the unperturbed system \eqref{eqn:unperturbed}.
We recall that for  a given homoclinic channel $\Gamma_0$,
the corresponding scattering map  $\sigma_0$,
is a phase-shift of the form
\[\sigma_0(I,\theta)=(I,\theta+\Delta(I)).\]

We choose and fix an energy level $h$ of $H_0$, and a point $z_0\in\Gamma_0\cap M_h$.
In the $(I^{\un,\st},\theta^{\un,\st},y^{\un,\st},x^{\un,\st})$-coordinates, $z_0$ is given by
\begin{equation}\label{eqn:notation1}
z_0=(I^\st_0,\theta^\st_0,y^\st_0,0)=(I^\un_0,\theta^\un_0,0,x^\un_0),\end{equation}
where $I^\st_0=I^\un_0=I_0$.
The effect of the flow $\Phi^\tau_0$ on $z_0$ in these coordinates is given by
\begin{equation}\label{eqn:notation2}
\Phi^\tau_0(z_0)=(I^\st_0,\theta^\st_0+\omega(I_0)\tau,y^\st(\tau),0)=
(I^\un_0,\theta^\un_0+\omega(I_0)\tau,0,x^\un(\tau)),\end{equation}
where $y^\st(\tau)$ and $x^\un(\tau)$ are the $y^\st$-component and  the $x^\un$-component, respectively, of $\Phi^\tau_0(z_0)$ evaluated at time $\tau$, and $\omega(I_0)=\frac{\partial h_0}{\partial I}$.

There exist uniquely defined points $z^-_0$, $z^+_0$ in $\lambda_0(h)$ such that $W^\un(z^-_0)\cap (\Gamma_0\cap M_h)= W^\st(z^+_0)\cap (\Gamma_0\cap M_h)=\{z_0\}$.
In the $(I,\theta,y,x)$-coordinates, the foot-points $z_0^\pm$ are given by
\begin{equation}\label{eqn:notation3}\begin{split}
z^-_0=(I_0,\theta^-_0,0,0),\quad
z^+_0= (I^\st_0,\theta^+_0,0,0),
\end{split}
\end{equation}
where $\theta^-_0=\theta^\un_0$ and $\theta^+_0=\theta^\st_0$.
The effect of the flow $\Phi^\tau_0$ on $z_0^\pm$ in these coordinates is given by
\begin{equation}\label{eqn:notation4}\begin{split}
\Phi^\tau_0(z^-_0)=(I_0,\theta^-_0+\omega(I_0)\tau,0,0),\\
\Phi^\tau_0(z^+_0)= (I_0,\theta^+_0+\omega(I_0)\tau,0,0).
\end{split}
\end{equation}

The scattering map $\sigma_0$ takes $z^-_0\in\lambda_{0}(h)$ into $z^+_0\in\lambda_{0}(h)$.

In the extended system \eqref{eqn:generalperturbation_t}, the corresponding homoclinic point is $\tz_0=(z_0,t_0)$ for some $t_0\in\R$.
The scattering map $\tilde{\sigma}_0$
takes  $\tz^-_0=(z^-_0,t_0)$ into $\tz^+_0=(z^+_0,t_0)$.

We will compute the effect of the perturbation on the scattering map $\tilde{\sigma}_0$.

When we add the perturbation there exists a homoclinic point $\tilde{z}_\eps\in\tGamma_\eps$ corresponding  to $\tilde{z}_0=(z_0,t_0)$ from the unperturbed case, such that $\tz_\eps$ satisfies the condition \textbf{(A-iii-b)}.
 Associated to $\tilde{z}_\eps\in\tGamma_\eps$ we have the points
$\tilde{z}^-_\eps$, $\tilde{z}^+_\eps$ in $\tilde{\Lambda}_\eps$ such that $W^\un(\tilde{z}^-_\eps)\cap \tilde{\Gamma}_\eps=\{\tilde{z}_\eps\}$, and $W^\st(\tilde{z}^+_\eps)\cap \tilde{\Gamma}_\eps=\{\tilde{z}_\eps\}$. The scattering map $\tilde\sigma_\eps$ takes $\tilde{z}^-_\eps\in\tLambda_\eps$ into $\tilde{z}^+_\eps\in\tLambda_\eps$.

In the sequel we will  make  a quantitative comparison between
\[\tilde{z}^-_0\mapsto \tilde{\sigma}_0({z}^-_0):=\tilde{z}^+_0,\]
and
\[\tilde{z}^-_\eps\mapsto \tilde{\sigma}_\eps(\tilde{z}^-_\eps):=\tilde{z}^+_\eps.\]

\subsubsection{Estimates}
Below we will refer to the notation in \eqref{eqn:notation1}, \eqref{eqn:notation2}, \eqref{eqn:notation3}, and \eqref{eqn:notation4}.
To simplify notation, we denote
$I^{\st}_\eps=I^\st(\tz_\eps)$,
$I^{\un}_\eps=I^\un(\tz_\eps)$,
$I^{\st +}_\eps=I^\st(\tz^+_\eps)$,
$I^{\un -}_\eps=I^\un(\tz^-_\eps)$,
$\xi^{\st}_\eps=(x^\st y^\st)(\tz_\eps)$,
$\xi^{\st+}=(x^\st y^\st)(\tz^+_\eps)$,
$\xi^{\un}_\eps=(x^\un y^\un)(\tz_\eps)$,
$\xi^{\un-}=(x^\un y^\un)(\tz^-_\eps)$.

Note that in the following, the coordinates of the scattering map can
be considered as functions of the point.  Hence, the symbols $O(\eps)$
can be interpreted as relative to the $C^r$ norm.

\begin{lem}\label{lem:estimates}
\begin{itemize}
\item[(i)] Estimates on $I$:
\begin{equation}\label{eqn:diff_I_estimates}
\begin{split}
I^{\st+}_\eps-I^{\st}_\eps=O(\eps),\,
I^{\un-}_\eps-I^{\un}_\eps=O(\eps),\,
I^{\st+}_\eps-I^{\un-}_\eps=O(\eps).
\end{split}
\end{equation}
\item[(ii)] Estimates on $h_0$:
\begin{equation}\label{eqn:h_0_estimates}\begin{split}
h_0(I^{\st+}_\eps)-h_0(I^{\st}_\eps)=&(I^{\st+}_\eps - I^{\st}_\eps)\left( \frac{\partial h_0}{\partial I}(I_0)\right) +O(\eps^2),\\
h_0(I^{\un-}_\eps)-h_0(I^{\un}_\eps)=&(I^{\un-}_\eps - I^{\un}_\eps)\left( \frac{\partial h_0}{\partial I}(I_0)\right) +O(\eps^2).
  \end{split}
  \end{equation}
  \item[(iii)] Estimates on $\xi$:
\begin{equation}\label{eqn:xi_estimates}\xi^{\st}_\eps=O(\eps),\, \xi^{\un}_\eps=O(\eps),\, \xi^{\st+}_\eps=O(\eps^2),\,  \xi^{\un-}_\eps=O(\eps^2).
\end{equation}

\item[(iv)] Estimates on $g_1$:
\begin{equation}\label{eqn:g1_estimates}\begin{split}
  g_1(I^{\st+}_\eps)=&g_1(I_0)+\frac{\partial g_1}{\partial I}(I_0)(I^{\st+}_\eps-I_0)+O(\eps^2),\\
  g_1(I^{\un-}_\eps)=&g_1(I_0)+\frac{\partial g_1}{\partial I}(I_0)(I^{\un-}_\eps-I_0)+O(\eps^2),
\end{split}
\end{equation}
\item[(v)] Estimates on $\frac{\partial g_1}{\partial I}$:
\begin{equation}\label{eqn:partial_g1_estimates}\begin{split}
  \frac{\partial g_1}{\partial I}(I^{\st+}_\eps)=&\frac{\partial g_1}{\partial I}(I_0)+\frac{\partial^2 g_1}{\partial I^2}(I_0)(I^{\st+}_\eps-I_0)+O(\eps^2),\\
  \frac{\partial g_1}{\partial I}(I^{\un-}_\eps)=&\frac{\partial g_1}{\partial I}(I_0)+\frac{\partial^2 g_1}{\partial I^2}(I_0)(I^{\un-}_\eps-I_0)+O(\eps^2),
\end{split}
\end{equation}
\end{itemize}
\end{lem}

\begin{proof}

(i)  Since $\tz_\eps=\tz_0+O(\eps)$ and $\tz^\pm_\eps=\tz^\pm_0+O(\eps)$, we have
$I^{\st,\un}(\tz_\eps)=I^{\st,\un}(\tz_0)+O(\eps)$, and
$I^{\st,\un}(\tz^\pm_\eps)=I^{\st,\un}(\tz^\pm_0)+O(\eps)$.
The fact that $I^{\st,\un}(\tz_0)=I^{\st,\un}(\tz^\pm_0)=I_0$  yields the estimates in  (i).

(ii) We estimate the term $h_0(I^{\st+}_\eps)-h_0(I^{\st}_\eps)$.
Applying the integral form of the Mean Value Theorem we have
\begin{equation}\label{eqn:diff_h_0}\begin{split}
 h_0(I^{\st+}_\eps)-h_0(I^{\st}_\eps)=&(I^{\st+}_\eps - I^{\st}_\eps)\int_{0}^{1}\frac{\partial h_0}{\partial I}(t I^{\st+}_\eps+(1-t)I^{\st}_\eps)dt.
\end{split}\end{equation}

We write the integrand of \eqref{eqn:diff_h_0} as a Taylor expansion
\begin{equation}\label{eqn:MVT_h_0}\begin{split}
 \frac{\partial h_0}{\partial I}(t I^{\st+}_\eps+(1-t)I^{\st}_\eps)=&\frac{\partial h_0}{\partial I}(I_0)\\&+
 \frac{\partial^2 h_0}{\partial I^2}(I_0)(t (I^{\st+}_\eps-I^{\st+}_0)+(1-t)(I^{\st}_\eps-I^{\st}_0))\\&+O(\eps^2)
\end{split}\end{equation}
where we used
\[\frac{\partial h_0}{\partial I}(t I^{\st+}_0+(1-t)I^{\st}_0)=\frac{\partial h_0}{\partial I}(I_0)\textrm {
and }
\frac{\partial^2 h_0}{\partial I^2}(t I^{\st+}_0+(1-t)I^{\st}_0)=\frac{\partial^2 h_0}{\partial I^2}(I_0).\]

Since   $I^{\st+}_\eps-I^{\st+}_0=O(\eps)$ and $I^{\st}_\eps-I^{\st}_0=O(\eps)$, we  have
\begin{equation}\label{eqn:diff_h_0_order}
   \frac{\partial h_0}{\partial I}(t I^{\st+}_0+(1-t)I^{\st}_0)=\frac{\partial h_0}{\partial I}(I_0)+O(\eps).
\end{equation}
Since $I^{\st+}_\eps- I^{\st}_\eps=O(\eps)$ from \eqref{eqn:diff_h_0} we obtain the first estimate in \eqref{eqn:h_0_estimates}.

The other estimate follows similarly.

(iii) 
The fact that $\xi^{\st,\un}=O(\eps)$ follows from
$\tz_\eps=\tz_0+O(\eps)$, and
$\xi^{\st,\un}(\tz_0)=0$, hence $\xi^{\st,\un}(\tz_\eps)=\xi^{\st,\un}(\tz_0)+O(\eps)=O(\eps)$.

In the same way we obtain $\xi^{\st+}=\xi^{\un-}=O(\eps)$. To prove that in fact  $\xi^{\st+}=\xi^{\un-}=O(\eps^2)$, we proceed as follows.
For $\xi^{\st+}_\eps$, we use the Taylor expansion:
\begin{equation*}\label{eqn:xi_s_taylor}
\begin{split}
   \xi^{\st+}_\eps=&\xi^{\st+}_0 +D\xi^{\st+}_0\cdot (\xi^{\st+}_\eps-\xi^{\st+}_0)+O(\eps^2),\\
\end{split}
\end{equation*}
where $\cdot$ denotes the dot product, and we used that $\xi^{\st+}_\eps-\xi^{\st+}_0=O(\eps)$.
We have \[\xi^{\st+}_0=0\] and
\[D\xi^{\st+}_0\cdot (\xi^{\st+}_\eps-\xi^{\st+}_0)=x^\st(\tz^+_0)(y^{\st+}_\eps-y^{\st+}_0)+y^\st(\tz^+_0)(x^{\st+}_\eps-x^{\st+}_0)=0,\]
since $y^\st(\tz^+_0)=x^\st(\tz^+_0)=0$ on $\tLambda_0$.
Therefore \eqref{eqn:xi_s_taylor} implies
\begin{equation*}\label{eqn:xi_s_taylor_order}
\begin{split}
   \xi^{\st+}_\eps=&O(\eps^2).
\end{split}
\end{equation*}
Similarly we obtain
\begin{equation*}\label{eqn:xi_u_taylor_order}
\begin{split}
   \xi^{\un-}_\eps=&O(\eps^2).
\end{split}
\end{equation*}

(iv) We write $g_1(I^{\st+}_\eps)$ as a Taylor expansion, using   that  $I^{\st+}_\eps-I^{\st+}_0=O(\eps)$, obtaining
\begin{equation*}\label{eqn:g1_s_taylor}
 g_1(I^{\st+}_\eps)=g_1(I^{\st+}_0)+\frac{\partial g_1}{\partial I}(I^{\st+}_0)(I^{\st+}_\eps-I^{\st+}_0)+O(\eps^2).
\end{equation*}
Since   $I^{\st+}_0=I_0$, the first equation in \eqref{eqn:g1_estimates} follows.

Similarly
\begin{equation*}\label{eqn:g1_u_taylor}
  g_1(I^{\un-}_\eps)=g_1(I^{\un-}_0)+\frac{\partial g_1}{\partial I}(I^{\un-}_0)(I^{\un-}_\eps-I^{\un-}_0)+O(\eps^2),
\end{equation*}
and $I^{\un-}_0=I_0$ yield the second equation in \eqref{eqn:g1_estimates}.

(v) The proof is similar to that of (iv).
\end{proof}

\subsubsection{Change in action by the scattering map}
\label{section:change_in_I}
We now give the expression of the action-component $\tilde{\mathcal{S}}^I$ of the mapping $\tilde{\mathcal{S}}$ in \eqref{eqn:main_sigma}.
Below we will refer to the notation in \eqref{eqn:notation1}, \eqref{eqn:notation2}, \eqref{eqn:notation3}, and \eqref{eqn:notation4}.

\begin{prop}\label{prop:PCRTBP_change_in_I}
The change in $I$ by the scattering map $\tilde{\sigma}_\eps$ is given by
the following formula
\begin{equation}\label{eqn:PCRTBP_I_plus_minus}
\begin{split}
  I\left(\tilde{z}^{+}_{\eps}\right)
  &-I\left( \tilde{z}^{-}_{\eps}\right)\\
 =&-\eps\left(\frac{\partial h_0}{\partial I}(I_0)\right)^{-1}\int_{0}^{+\infty}\left((\X^{1}H_0)
 ({\Phi}^{\tau}_{0}(\tilde{z}^+_{0})
  -(\X^{1}H_0)(\tilde{\Phi}^{\tau}_{0}(\tilde{z}_{0}))\right)d\tau
\\&-\eps\left(\frac{\partial h_0}{\partial I}(I_0)\right)^{-1}\int_{-\infty}^{0}\left((\X^{1}H_0)(\tilde{\Phi}^{\tau}_{0}(\tilde{z}^-_{0}))
  -(\X^{1}H_0)(\tilde{\Phi}^{\tau}_{0}(\tilde{z}_{0}))\right)d\tau\\
  &+O\left(\eps^{2}\right).
  \end{split}
\end{equation}
When   the perturbation is Hamiltonian $\X^1=J\nabla H_1$, in \eqref{eqn:PCRTBP_I_plus_minus}
we have  $\X^{1}H_0=\{H_0,H_1\}$, where $\{\cdot,\cdot\}$ denotes the Poisson bracket.
\end{prop}

Proposition \ref{prop:PCRTBP_change_in_I} implies that
\[\tilde\sigma^I_\eps(I,\theta,t)=\tilde\sigma^I_0(I,\theta,t)+\eps \tilde{\mathcal{S}}^I (I,\theta,t)+O(\eps^2),\]
where $\eps \tilde{\mathcal{S}}^I$ is given by the first two terms on the right hand-side of \eqref{eqn:PCRTBP_I_plus_minus}.  The expression of $\eps \tilde{\mathcal{S}}^I$ is particularly simple since $I$ is a slow variable.

We note that we can express the right-hand side of \eqref{eqn:PCRTBP_I_plus_minus} in terms of the $(I^{\un,\st},\theta^{\un,\st},y^{\un,\st},x^{\un,\st})$ coordinates, by making the  following substitutions:
   \begin{equation*}
     \begin{split}
     \tilde{\Phi}^{\tau}_{0}(\tilde{z}^+_{0})=&\left(I^\st_0, \theta^\st_0+\omega(I_0)\tau,0,0,t_0+\tau\right),\\
     \tilde{\Phi}^{\tau}_{0}(\tilde{z}^-_{0})=&\left(I^\st_0, \theta^\un_0+\omega(I_0)\tau,0,0,t_0+\tau\right),\\
     \tilde{\Phi}^{\tau}_{0}(\tilde{z}_{0})=&\left(I^\st_0, \theta^\st_0+\omega(I_0)\tau,y^\st(\tau),0, t_0+\tau \right)\\
     =&\left(I^\un_0, \theta^\un_0+\omega(I_0)\tau,0,x^\un(\tau),t_0+\tau \right).
     \end{split}
   \end{equation*}

To prove Proposition \ref{prop:PCRTBP_change_in_I} we will use the following:
\begin{lem}\label{lem:change_in_H}
The change in $H_0$ by the scattering map $\tilde{\sigma}_\eps$ is given by
the following equation
\begin{equation}\label{eqn:change_in_H}\begin{split}
H_0(\tz^+_\eps)-H_0(\tz^-_\eps)=&-\eps\int_{0}^{+\infty}\left((\X^{1}H_0)
(\tilde{\Phi}^{\tau}_{0}(\tilde{z}^+_{0}))
  -(\X^{1}H_0)(\tilde{\Phi}^{\tau}_{0}(\tilde{z}_{0}))\right)d\tau
\\&-\eps\int_{-\infty}^{0}\left((\X^{1}H_0)(\tilde{\Phi}^{\tau}_{0}(\tilde{z}^-_{0}))
  -(\X^{1}H_0)(\tilde{\Phi}^{\tau}_{0}(\tilde{z}_{0}))\right)d\tau\\
  &+O\left(\eps^{2}\right).
\end{split}
\end{equation}
\end{lem}
\begin{proof}[Proof of Lemma \ref{lem:change_in_H}]
First, note that \[(\X^0+\eps \X^1)H_0=\X^0H_0+\eps \X^1H_0=\{H_0,H_0\}+\eps \X^1H_0=\eps \X^1H_0.\]

Second, applying Lemma \ref{lem:master_1} and Lemma \ref{lem:master_2} from Appendix \ref{sec:master}, for $\bF=H_0$, we have
\begin{equation}\label{eqn:H0s_int}\begin{split}
H_0(\tz^+_\eps)-H_0(\tz_\eps)=&-\eps\int_{0}^{+\infty}\left((\X^{1}H_0)(\tilde{\Phi}^{\tau}_{0}
(\tilde{z}^+_{0}))
  -(\X^{1}H_0)(\tilde{\Phi}^{\tau}_{0}(\tilde{z}_{0}))\right)d\tau\\
 &+O\left(\eps^{1+\varrho}\right)
\end{split}\end{equation}
\begin{equation}\label{eqn:H0u_int} \begin{split}
H_0(\tz^-_\eps)-H_0(\tz_\eps)=&\eps\int_{-\infty}^{0}\left((\X^{1}H_0)
(\tilde{\Phi}^{\tau}_{0}(\tilde{z}^-_{0}))
  -(\X^{1}H_0)(\tilde{\Phi}^{\tau}_{0}(\tilde{z}_{0}))\right)d\tau\\
  &+O\left(\eps^{1+\varrho}\right).
\end{split}\end{equation}
Subtracting the two equations from above, after cancelling out the common term $H_0(\tz_\eps)$ representing the value of $H_0$ at the homoclinic point $\tz_0$, we obtain  \eqref{eqn:change_in_H} with an error term of order $O(\eps^{1+\varrho})$. Since the function $H_0(\tz^+_\eps)-H_0(\tz^-_\eps)$ can be expanded as a Taylor series in $\eps$, by matching the corresponding terms of this Taylor  expansion with the terms in \eqref{eqn:H0s_int} minus \eqref{eqn:H0u_int}, it follows that the error term  $O(\eps^{1+\varrho})$    must equal $O(\eps^2)$.
\end{proof}

\begin{proof}[Proof of Proposition \ref{prop:PCRTBP_change_in_I}]


By Proposition \ref{prop:normal_form} we have
\begin{equation}\label{eqn:H0s}\begin{split}
H_0(\tz^+_\eps)=&h_0(I^{\st+}_\eps)+(\xi^{\st+}_\eps) g_1(I^{\st+}_\eps)+(\xi^{\st+}_\eps)^2g_2(I^{\st+}_\eps,\xi^{\st+}_\eps),
\end{split}\end{equation}
\begin{equation}\label{eqn:H0u}\begin{split}
H_0(\tz^-_\eps)=&h_0(I^{\un-}_\eps)+(\xi^{\un-}_\eps) g_1(I^{\un-}_\eps)+(\xi^{\un-}_\eps)^2g_2(I^{\un-}_\eps,\xi^{\un-}_\eps).
\end{split}\end{equation}
Subtracting we obtain
\begin{equation}\label{eqn:diff_H_0_h_0}\begin{split}
H_0(\tz^+_\eps)-H_0(\tz^-_\eps)=&(h_0(I^{\st+}_\eps)-h_0(I^{\un-}_\eps))\\&+((\xi^{\st+}_\eps) g_1(I^{\st+}_\eps)-(\xi^{\un-}_\eps) g_1(I^{\un-}_\eps))\\&+O(\eps^2)
\end{split}\end{equation}
where the error term $O(\eps^2)$ in the above is due to Lemma \ref{lem:estimates} equation \eqref{eqn:xi_estimates}.

The term $h_0(I^{\st+}_\eps)-h_0(I^{\un-}_\eps)$ in \eqref{eqn:diff_H_0_h_0} is given, by Lemma \ref{lem:estimates} equation \eqref{eqn:h_0_estimates}, as
\begin{equation}\label{eqn:diff_h_0_OO}
  h_0(I^{\st+}_\eps)-h_0(I^{\un-}_\eps)=(I^{\st+}_\eps - I^{\un-}_\eps)\left( \frac{\partial h_0}{\partial I}(I_0)\right) +O(\eps^2).
\end{equation}

 Since, by Lemma \ref{lem:estimates} equation \eqref{eqn:xi_estimates}, we have
\begin{equation}\label{eqn:xi__taylor_order}
\begin{split}
   \xi^{\st+}_\eps=O(\eps^2),\, \xi^{\un-}_\eps=O(\eps^2).
\end{split}
\end{equation}
we obtain
\begin{equation}(\xi^{\st+}_\eps) g_1(I^{\st+}_\eps)-(\xi^{\un-}_\eps) g_1(I^{\un-}_\eps)=O(\eps^2).\end{equation}



Thus, from \eqref{eqn:diff_H_0_h_0} and using \eqref{eqn:diff_h_0_OO} we have
\begin{equation}\label{diff_H_0_OO}
  H_0(\tz^+_\eps)-H_0(\tz^-_\eps)=(I^{\st+}_\eps - I^{\un-}_\eps)\left( \frac{\partial h_0}{\partial I}(I)\right) +O(\eps^2).
\end{equation}

As the left-hand side of \eqref{diff_H_0_OO} is given by \eqref{eqn:change_in_H}, since $\frac{\partial h_0}{\partial I}\neq 0$ by condition
\textbf{(A-iii-a)},
solving for $I^{\st+}_\eps - I^{\un-}_\eps$ yields
\begin{equation}\label{eqn:I_eps_change}
  \begin{split}
     I^{\st+}_\eps - I^{\un-}_\eps =&\left( \frac{\partial h_0}{\partial I}(I)\right)^{-1}\left(H_0(\tz^+_\eps)-H_0(\tz^-_\eps)\right) \\
       =&-\eps\left(\frac{\partial h_0}{\partial I}(I)\right)^{-1}\int_{0}^{+\infty}\left((\X^{1}H_0)\left(\tilde{\Phi}^{\tau}_{0}\left(\tilde{z}^+_{0}\right)\right)
  -(\X^{1}H_0)\left(\tilde{\Phi}^{\tau}_{0}\left(\tilde{z}_{0}\right)\right)\right)d\tau
\\&-\eps\left(\frac{\partial h_0}{\partial I}(I)\right)^{-1}\int_{-\infty}^{0}\left((\X^{1}H_0)\left(\tilde{\Phi}^{\tau}_{0}\left(\tilde{z}^-_{0}\right)\right)
  -(\X^{1}H_0)\left(\tilde{\Phi}^{\tau}_{0}\left(\tilde{z}_{0}\right)\right)\right)d\tau\\
  &+O\left(\eps^{1+\varrho}\right).
  \end{split}
\end{equation}

By the same argument as in the proof of Lemma \ref{lem:change_in_H}, the error term  $O(\eps^{1+\varrho})$ in \eqref{eqn:I_eps_change}  must equal $O(\eps^2)$.  This shows \eqref{eqn:PCRTBP_I_plus_minus}.
\end{proof}

\begin{prop}
\begin{equation}\label{eqn:change_I_s_int_0}\begin{split}
I^{\st+}_\eps-&I^{\st}_\eps=\eps\left(\frac{\partial h_0}{\partial I}(I_0)) \right)^{-1}\int^{+\infty}_{0}\left((\X^{1}H_0)(\tilde{\Phi}^{\tau}_{0}
(\tilde{z}^+_{0}))
-(\X^{1}H_0)(\tilde{\Phi}^{\tau}_{0}(\tilde{z}_{0}))\right)d\tau\\
  &-\eps \left(\frac{\partial h_0}{\partial I}(I_0)) \right)^{-1} g_1(I_0)\int^{+\infty}_{0}\left((\X^{1}\xi^\st)(\tilde{\Phi}^{\tau}_{0}
(\tilde{z}^+_{0})) -(\X^{1}\xi^\st)(\tilde{\Phi}^{\tau}_{0}(\tilde{z}_{0}))\right)d\tau\\
 &+O\left(\eps^{2}\right).
 \end{split}\end{equation}

\begin{equation}\label{eqn:change_I_u_int_0}\begin{split}
I^{\un-}_\eps-&I^{\un}_\eps=\eps\left(\frac{\partial h_0}{\partial I}(I_0)) \right)^{-1}\int_{-\infty}^{0}\left((\X^{1}H_0)(\tilde{\Phi}^{\tau}_{0}
(\tilde{z}^-_{0}))
-(\X^{1}H_0)(\tilde{\Phi}^{\tau}_{0}(\tilde{z}_{0}))\right)d\tau\\
  &-\eps \left(\frac{\partial h_0}{\partial I}(I_0)) \right)^{-1} g_1(I_0)\int_{-\infty}^{0}\left((\X^{1}\xi^\un)(\tilde{\Phi}^{\tau}_{0}
(\tilde{z}^-_{0})) -(\X^{1}\xi^\un)(\tilde{\Phi}^{\tau}_{0}(\tilde{z}_{0}))\right)d\tau\\
 &+O\left(\eps^{2}\right).
\end{split}\end{equation}
\end{prop}

\begin{proof}
The formula for
$H_0(\tz^+_\eps)-H_0(\tz_\eps)$ is given by \eqref{eqn:H0s_int}.
The normal form expansion of $H_0$ at a homoclinic point $\tz_\eps$ can be written with respect to the two sets of coordinates as
\begin{equation}\label{eqn:H0s_hom}
H_0(\tz _\eps)=h_0(I^{\st }_\eps)+(\xi^{\st }_\eps) g_1(I^{\st }_\eps)+(\xi^{\st }_\eps)^2g_2(I^{\st }_\eps,\xi^{\st }_\eps)\end{equation}
\begin{equation}\label{eqn:H0u_hom}
H_0(\tz _\eps)=  h_0(I^{\un}_\eps)+(\xi^{\un }_\eps) g_1(I^{\un }_\eps)+(\xi^{\un}_\eps)^2g_2(I^{\un }_\eps,\xi^{\un }_\eps)
\end{equation}
Subtracting \eqref{eqn:H0s_hom} from \eqref{eqn:H0s} we obtain
\begin{equation}\label{eqn:H0s_hom_diff}\begin{split}
H_0(\tz^+ _\eps)-H_0(\tz _\eps)=& h_0(I^{\st+}_\eps)-h_0(I^{\st }_\eps)+(\xi^{\st+}_\eps) g_1(I^{\st+}_\eps)-(\xi^{\st }_\eps) g_1(I^{\st }_\eps)+O(\eps^2)\\
=& \left (\frac{\partial h_0}{\partial I}(I_0)\right)(I^{\st+}_\eps - I^{\st }_\eps) \\ &+(\xi^{\st+}_\eps)(g_1(I^{\st+}_\eps)-g_1(I^{\st }_\eps))\\&+ (\xi^{\st+}_\eps-\xi^{\st }_\eps) g_1(I^{\st}_\eps) +O(\eps^2)\\
=&\left (\frac{\partial h_0}{\partial I}(I_0)\right)(I^{\st+}_\eps - I^{\st }_\eps) \\
&+ (\xi^{\st+}_\eps-\xi^{\st }_\eps) g_1(I_0) +O(\eps^2)
\end{split}\end{equation}
In the above we have used  Lemma \ref{lem:estimates}, equations \eqref{eqn:h_0_estimates}, \eqref{eqn:xi_estimates}, and \eqref{eqn:g1_estimates}.

Applying Lemma \ref{lem:master_1} and Lemma \ref{lem:master_2} from Appendix \ref{sec:master}, for $\bF=\xi^\st$, we have
\begin{equation}\label{eqn:xi_s_int}\begin{split}
\xi^{\st+}_\eps-\xi^{\st}_\eps=&-\eps\int_{0}^{+\infty}\left((\X^{1}\xi^\st)\left(\tilde{\Phi}^{\tau}_{0}
\left(\tilde{z}^+_{0}\right)
\right)
  -(\X^{1}\xi^\st)\left(\tilde{\Phi}^{\tau}_{0}\left(\tilde{z}_{0}\right)\right)\right)d\tau\\
 &+O\left(\eps^{1+\varrho}\right).
\end{split}\end{equation}

Thus, using \eqref{eqn:H0s_hom_diff}, \eqref{eqn:H0s_int}, and \eqref{eqn:xi_s_int}
we obtain \eqref{eqn:change_I_s_int_0}.

Equation \eqref{eqn:change_I_u_int_0}  follows similarly.

The argument that the error term $O(\eps^{1+\varrho})$ can be replaced by $O(\eps^2)$ follows in the same way as in
the proof of Lemma \ref{lem:change_in_H}.
\end{proof}

Let $\tau\in\R$ be some value of the time variable.
Applying the formula \eqref{eqn:change_I_s_int_0} to $\tPhi^\tau_\eps(\tilde{z}^{+}_{\eps})$ and $\tPhi^\tau_\eps(\tilde{z}_{\eps})$ instead $\tilde{z}^{+}_{\eps}$ and $\tilde{z}_{\eps}$, respectively,
and the formula \eqref{eqn:change_I_u_int_0}  to $\tPhi^\tau_\eps(\tilde{z}^{-}_{\eps})$ and $\tPhi^\tau_\eps(\tilde{z}_{\eps})$ instead $\tilde{z}^{-}_{\eps}$ and $\tilde{z}_{\eps}$, respectively,
we obtain:

\begin{cor}\label{cor:cor} For any time $\tau\in \R$ we have
\begin{equation}\label{eqn:PCTBP_I_plus_sigma_half}\begin{split}
I^{\st}_\eps(\tPhi^\tau(\tz_\eps^+))-&I^{\st}_\eps(\tPhi^\tau(\tz_\eps))\\=&
-\eps\left(\frac{\partial h_0}{\partial I}(I_0)) \right)^{-1}\int_{0}^{+\infty}\left((\X^{1}H_0)\left(\tilde{\Phi}^{\tau+\varsigma}_{0}
\left(\tilde{z}^+_{0}\right)\right)\right.\\
&\qquad\qquad\qquad\qquad\left.  -(\X^{1}H_0)\left(\tilde{\Phi}^{\tau+\varsigma}_{0}\left(\tilde{z}_{0}\right)\right)\right)d\varsigma\\
  &+\eps \left(\frac{\partial h_0}{\partial I}(I_0)) \right)^{-1} g_1(I_0)\int_{0}^{+\infty}\left((\X^{1}\xi^\st)\left(\tilde{\Phi}^{\tau+\varsigma}_{0}
\left(\tilde{z}^+_{0}\right)
\right)\right.\\
&\qquad\qquad\qquad\qquad\left.  -(\X^{1}\xi^\st)\left(\tilde{\Phi}^{\tau+\varsigma}_{0}\left(\tilde{z}_{0}\right)\right)\right)d\varsigma\\
 &+O\left(\eps^{2}\right).\end{split}\end{equation}
 \begin{equation}\label{eqn:PCTBP_I_minus_sigma_half}\begin{split}
I^{\un}_\eps(\tPhi^\tau(\tz_\eps^-))-&I^{\un}_\eps(\tPhi^\tau(\tz_\eps))\\=&
\eps\left(\frac{\partial h_0}{\partial I}(I_0)) \right)^{-1}\int_{-\infty}^{0}\left((\X^{1}H_0)\left(\tilde{\Phi}^{\tau+\varsigma}_{0}
\left(\tilde{z}^-_{0}\right)
\right)\right.\\
&\qquad\qquad\qquad\qquad\left.   -(\X^{1}H_0)\left(\tilde{\Phi}^{\tau+\varsigma}_{0}\left(\tilde{z}_{0}\right)\right)\right)d\varsigma\\
  &-\eps \left(\frac{\partial h_0}{\partial I}(I_0)) \right)^{-1}g_1(I_0)\int_{-\infty}^{0}\left((\X^{1}\xi^\un)\left(\tilde{\Phi}^{\tau+\varsigma}_{0}
\left(\tilde{z}^-_{0}\right)
\right)\right.\\
&\qquad\qquad\qquad\qquad\qquad\left.
  -(\X^{1}\xi^\un)\left(\tilde{\Phi}^{\tau+\varsigma}_{0}\left(\tilde{z}_{0}\right)\right)\right)d\varsigma\\
 &+O\left(\eps^{2}\right).
\end{split}\end{equation}
\end{cor}

\begin{cor}\label{cor:cancellation}
\begin{equation}\label{eqn:cancellation}\begin{split}
&\int_{0}^{+\infty}\left((\X^{1}\xi^\st)(\tilde{\Phi}^{\tau}_{0}
(\tilde{z}^+_{0})) -(\X^{1}\xi^\st)(\tilde{\Phi}^{\tau}_{0}(\tilde{z}_{0}))\right)d\tau\\
&\qquad -\int_{-\infty}^{0}\left((\X^{1}\xi^\un)(\tilde{\Phi}^{\tau}_{0}
(\tilde{z}^-_{0})) -(\X^{1}\xi^\un)(\tilde{\Phi}^{\tau}_{0}(\tilde{z}_{0}))\right)d\tau=0.
\end{split}\end{equation}
\end{cor}
\begin{proof}
Let us denote
\begin{equation}\label{eqn:J}
\begin{split}J^{+}=&\int_{0}^{+\infty}\left((\X^{1}\xi^\st)(\tilde{\Phi}^{\tau}_{0}
(\tilde{z}^+_{0})) -(\X^{1}\xi^\st)(\tilde{\Phi}^{\tau}_{0}(\tilde{z}_{0}))\right)d\tau,\\
J^{-}=&-\int_{-\infty}^{0}\left((\X^{1}\xi^\un)(\tilde{\Phi}^{\tau}_{0}
(\tilde{z}^-_{0})) -(\X^{1}\xi^\un)(\tilde{\Phi}^{\tau}_{0}(\tilde{z}_{0}))\right)d\tau.
\end{split}\end{equation}

Recall that, by condition \textbf{(A-iii-b)} and \eqref{eqn:theta_st_minus_theta_un}, we have $I^{\un}_\eps=I^{\st}_\eps$.
Subtracting \eqref{eqn:PCTBP_I_minus_sigma_half} from \eqref{eqn:PCTBP_I_plus_sigma_half}, and comparing with
\eqref{eqn:PCRTBP_I_plus_minus}, we should have
\[\left(\frac{\partial h_0}{\partial I}(I_0)) \right)^{-1} g_1(I_0)(J^+ + J^-)=0.\]
By condition \textbf{(A-iii-a)}, we have that $\frac{\partial h_0}{\partial I}(I_0)\neq 0$ and $g_1(I_0)\neq 0$,
therefore
\begin{equation}\label{eqn:J-cancellation}
J^+ + J^-=0.
\end{equation}
\end{proof}

\subsubsection{Change in angle by the scattering map}
\label{section:change_in_angle}

We now give the expression of the angle-component $\tilde{\mathcal{S}}^\theta$ of the mapping $\tilde{\mathcal{S}}$ in \eqref{eqn:main_sigma}.

\begin{prop}\label{prop:PCRTBP_change_in_theta}
The change in $\theta$ by the scattering map $\tilde{\sigma}_\eps$ is given by the following equation
\begin{equation}
\label{eqn:PCRTBP_change_in_theta}
\begin{split}
\theta^\st\left(\tilde{z}^{+}_{\eps}\right)-&\theta^\un\left( \tilde{z}_{\eps}^{-}\right)\\
    =&\Delta(I_0)-\varepsilon \int^{+\infty}_{0}\X^{1}\theta^\st(\tPhi^{\tau}_{0}(\tilde{z}^{+}_{0}))
-\X^{1}\theta^\st (\tPhi^{\tau}_{0}(\tilde{z}_{0})) d\tau\\
    &\qquad\,\, -\varepsilon \int_{-\infty}^{0}\X^{1}\theta^\un(\tPhi^{\tau}_{0}(\tilde{z}^{-}_{0}))
-\X^{1}\theta^\un (\tPhi^{\tau}_{0}(\tilde{z}_{0})) d\tau\\
    &+\eps\left(\frac{\partial^2 h_{0}}{\partial I^2}
(I_{0}) \right)\left(\frac{\partial h_0}{\partial I}(I_0) \right)^{-1}\int^{+\infty}_{0} \left((\X^{1}H_0)(\tilde{\Phi}^{\tau }_{0}
(\tilde{z}^+_{0}))\right.\\
&\qquad\qquad\qquad\qquad\qquad\qquad\qquad\qquad\left.
  -(\X^{1}H_0)(\tilde{\Phi}^{\tau }_{0}(\tilde{z}_{0}))\right)\tau d\tau\\
&+\eps\left(\frac{\partial^2 h_{0}}{\partial  I ^2}
(I_{0}) \right)\left(\frac{\partial h_0}{\partial I}(I_0) \right)^{-1}\int_{-\infty}^{0} \left((\X^{1}H_0)(\tilde{\Phi}^{\tau }_{0}
(\tilde{z}^-_{0}))\right.\\
&\qquad\qquad\qquad\qquad\qquad\qquad\qquad\qquad\left.
  -(\X^{1}H_0)(\tilde{\Phi}^{\tau }_{0}(\tilde{z}_{0}))\right)\tau d\tau\\
&+O(\eps^{2}),
\end{split}\end{equation}
where $\Delta(I_0)$ is  the phase-shift on the action level set $I_0$
that defines the unperturbed scattering map $\tilde{\sigma}_0$  (see Proposition \ref{prop:mismatch}).
\end{prop}

Proposition  \ref{prop:PCRTBP_change_in_theta} implies that
\[\tilde\sigma^\theta_\eps(I,\theta,t)=\tilde\sigma^\theta_0(I,\theta,t)+\Delta(I_0)+\eps \tilde{\mathcal{S}}^\theta (I,\theta,t)+O(\eps^2),\]
where $\eps \tilde{\mathcal{S}}^\theta$ is given by the first four terms on the right hand-side of \eqref{eqn:PCRTBP_change_in_theta}.  The expression of $\eps \tilde{\mathcal{S}}^\theta$  is more complicated than the one for $\eps \tilde{\mathcal{S}}^I$ since $\theta$ is a fast variable.

\begin{proof}
We will begin by computing the difference of the $\theta^{\st}$ evaluated at a homoclinic point and at the footpoint of the stable fiber through the homoclinic point. By Lemma \ref{lem:master_2} from Appendix \ref{sec:master}, we have

\[\theta^{\st}\left(\tilde{z}^{+}_{\varepsilon}\right)-\theta^{\st}\left(\tilde{z}_{\varepsilon}\right)=
-\int^{+\infty}_{0}\frac{d}{d\tau}\left[\theta^{\st}(\tPhi^{\tau}_{\varepsilon}(\tilde{z}^{+}_{\varepsilon}))
-\theta^{\st}(\tPhi^{\tau}_{\varepsilon}(\tilde{z}_{\varepsilon}))\right]d\tau\]

Now by equation \eqref{eqn:evolution4} we have:

\[\frac{d\theta^{\st}}{d\tau}=\frac{\partial H_{0}}{\partial I^{\st}}+\varepsilon \X^1 \theta^{\st}.\]
We can break the integral up into two parts:
\begin{equation}\label{eqn:A}
A=-\int^{+\infty}_{0}\frac{\partial H_{0}}{\partial I^{\st}}(\tPhi^{\tau}_{\varepsilon}(\tilde{z}^{+}_{\varepsilon}))
-\frac{\partial H_{0}}{\partial I^{\st}}(\tPhi^{\tau}_{\varepsilon}(\tilde{z}_{\varepsilon}))d\tau
\end{equation}
and
\begin{equation}\label{eqn:B}
B=-\varepsilon \int^{+\infty}_{0}\X^{1}\theta^\st(\tPhi^{\tau}_{\varepsilon}(\tilde{z}^{+}_{\varepsilon}))
-\X^{1}\theta^\st(\tPhi^{\tau}_{\varepsilon}(\tilde{z}_{\varepsilon})) d\tau.
\end{equation}
As for the integral $B$, $\eps\X^1\theta^\st$ is $O(\varepsilon)$. So, by Lemma \ref{lem:master_4} from Appendix \ref{sec:master}, we can express the integral in terms of the unperturbed system plus an error term:
\[B=-\varepsilon \int^{+\infty}_{0}\X^{1}\theta^\st(\tPhi^{\tau}_{0}(\tilde{z}^{+}_{0}))
-\X^{1}\theta^\st (\tPhi^{\tau}_{0}(\tilde{z}_{0})) d\tau+O(\varepsilon^{1+\varrho}).\]

Returning to the  integral \eqref{eqn:A}, we now use the normal form  \eqref{eqn:normal_st} of $H_{0}$ given by
Proposition \ref{prop:normal_form}, yielding
\begin{equation}\label{eqn:H_0_derivative}
  \frac{\partial H_{0}}{\partial I^{\st}}=\frac{\partial h_{0}}{\partial I^{\st}}(I^{\st})+(\xi^\st) \frac{\partial g_{1}}{\partial I^{\st}}(I^{\st})+(\xi^\st)^2\frac{\partial g_{2}}{\partial I^{\st}}(I^{\st},\xi^\st),
\end{equation}
where $\xi^\st=x^\st y^\st$.

Thus, the integral $A$ given by \eqref{eqn:A} breaks into three parts
\begin{eqnarray}\label{eqn:A1}
A_{1}=&\displaystyle -\int^{+\infty}_{0}\frac{\partial h_{0}}{\partial I^{\st}}( \tPhi^{\tau}_{\varepsilon}(\tilde{z}^{+}_{\varepsilon}))
-\frac{\partial h_{0}}{\partial I^{\st}}
( \tPhi^{\tau}_{\varepsilon}(\tilde{z}_{\varepsilon}))d\tau,\\[0.5em]
\label{eqn:A2}
A_{2}=&\displaystyle-\int^{+\infty}_{0}\xi^{\st}\frac{\partial g_{1}^\st}{\partial I^{\st}}(\tPhi^{\tau}_{\varepsilon}(\tilde{z}^{+}_{\varepsilon}))
-\xi^{\st}\frac{\partial g_{1}^\st}{\partial I^{\st}}(\tPhi^{\tau}_{\varepsilon}(\tilde{z}_{\varepsilon}))
d\tau ,\\[0.5em]
\label{eqn:A3}
A_{3}=&\displaystyle-\int^{+\infty}_{0}(\xi^{\st})^2\frac{\partial g_{2}^\st}{\partial I^{\st}}(\tPhi^{\tau}_{\varepsilon}(\tilde{z}^{+}_{\varepsilon}))
-(\xi^{\st})^2\frac{\partial g_{2}^\st}{\partial I^{\st}}
(\tPhi^{\tau}_{\varepsilon}(\tilde{z}_{\varepsilon}))d\tau .
\end{eqnarray}

From Lemma \ref{lem:estimates} equation \eqref{eqn:xi_estimates}
we have
$\xi^\st(\tPhi(\tz^+_\eps))=O(\varepsilon^2)$ and
$\xi^\st(\tPhi(\tz_\eps))=O(\varepsilon)$ in \eqref{eqn:A3}.
Thus, we can immediately obtain  that $A_{3}$ is $O(\varepsilon^{2})$.


We use the integral form of the Mean Value Theorem to rewrite the integral $A_{1}$.
Recall that\[\frac{\partial F}{\partial x}(b)-\frac{\partial F}{\partial x}(a)=(b-a)\int^{1}_{0}\frac{\partial^2 F}{\partial x^2}\left(a+t(b-a)\right)dt\]
Applying this result to $A_{1}$ for $b=I^{\st}( \tPhi^{\tau}_{\varepsilon}(\tilde{z}^{+}_{\varepsilon}))$,
$a=I^{\st}(\tPhi^{\tau}_{\varepsilon}(\tilde{z}_{\varepsilon}))$,
and $F=\frac{\partial h_{0}}{\partial I^{\st}}$, yields
\begin{equation}\label{eqn:A1_mean_value}
A_1=-\int^{+\infty}_{0} (I^{\st}(\tPhi^{\tau}_{\varepsilon}(\tilde{z}^{+}_{\varepsilon}))
-I^{\st}(\tPhi^{\tau}_{\varepsilon}(\tilde{z}_{\varepsilon}))) C_\eps(\tau) \,
d\tau,
\end{equation}
where $C_\eps$ stands for the integral
\begin{equation}\label{eqn:A4} C_\eps(\tau)=\int^{1}_{0}\frac{\partial^2 h_{0}}{\partial (I^{\st})^2}
\left[I^{\st}(\tPhi^{\tau}_{\varepsilon}(\tilde{z}_{\varepsilon}))
+t\left(I^{\st}(\tPhi^{\tau}_{\varepsilon}(\tilde{z}^{+}_{\varepsilon}))
-I^{\st}(\tPhi^{\tau}_{\varepsilon}(\tilde{z}_{\varepsilon}))\right)\right]dt.
\end{equation}

We  evaluate  the expression $I^{\st}(\tPhi^{\tau}_{\varepsilon}(\tilde{z}^{+}_{\varepsilon}))
-I^{\st}(\tPhi^{\tau}_{\varepsilon}(\tilde{z}_{\varepsilon}))$ in \eqref{eqn:A1_mean_value} by invoking Corollary \ref{cor:cor}, obtaining
\begin{equation}\begin{split}
&-\eps\left(\frac{\partial h_0}{\partial I}(I_0)) \right)^{-1}\int_{0}^{+\infty}\left((\X^{1}H_0)(\tilde{\Phi}^{\tau+\varsigma}_{0}
(\tilde{z}^+_{0}))
  -(\X^{1}H_0)(\tilde{\Phi}^{\tau+\varsigma}_{0}(\tilde{z}_{0}))\right)d\varsigma\\
  &+\eps \left(\frac{\partial h_0}{\partial I^\st}(I_0)) \right)^{-1} g_1(I_0)\int_{0}^{+\infty}\left((\X^{1}\xi^\st)(\tilde{\Phi}^{\tau+\varsigma}_{0}
(\tilde{z}^+_{0}))
  -(\X^{1}\xi^\st)(\tilde{\Phi}^{\tau+\varsigma}_{0}(\tilde{z}_{0}))\right)d\varsigma\\
 &+O\left(\eps^{2}\right).
\end{split}\end{equation}

The next part of the integrand is
\[C_{\varepsilon}(\tau)=\int^{1}_{0}\frac{\partial^2 h_{0}}{\partial (I^{\st})^2}
\left[I^{\st}\left(\tPhi^{\tau}_{\varepsilon}\left(\tilde{z}_{\varepsilon}\right)\right)
+t\left(I^{\st}\left(\tPhi^{\tau}_{\varepsilon}\left(\tilde{z}^{+}_{\varepsilon}\right)\right)-
I^{\st}\left(\tPhi^{\tau}_{\varepsilon}\left(\tilde{z}_{\varepsilon}\right)\right)\right)\right]dt.\]

Using Gronwall's inequality -- Lemma  \eqref{lem:Gronwall_application} --, we can write
\[C_{\varepsilon}(\tau)=C_{0}(\tau)+O(\varepsilon^{\varrho}),\]
where $0<\varrho<1$. However, when $\varepsilon=0$, $I^{\st}$ along the flow of the footpoint $\tilde{z}^{+}_0$ is equal to $I^{\st}$ along the flow of the homoclinic point $\tz_0$. Since $I^{\st}\left(\tilde{z}_{0}\right)=I_0$ we obtain
\[C_{0}(\tau)=\frac{\partial^2 h_{0}}{\partial I^2}
(I_0),\]
which is a constant.

Putting these expressions together, we can write $A_{1}$ as
\begin{equation}\label{eqn:double_integral}\begin{split}
&\eps\left(\frac{\partial^2 h_{0}}{\partial I^2}\left(I_0\right)\right) \left(\frac{\partial h_0}{\partial I}(I_0)) \right)^{-1}\int^{+\infty}_{0} \int_{0}^{+\infty}\left((\X^{1}H_0)(\tilde{\Phi}^{\tau+\varsigma}_{0}
(\tilde{z}^+_{0}))\right.\\
&\qquad\qquad\qquad\qquad\qquad\qquad\qquad\qquad\left.  -(\X^{1}H_0)(\tilde{\Phi}^{\tau+\varsigma}_{0}(\tilde{z}_{0}))\right)d\varsigma d\tau\\
  &-\eps\left( \frac{\partial^2 h_{0}}{\partial I^2}(I_0) \right)\left(\frac{\partial h_0}{\partial I}(I_0)) \right)^{-1} g_1(I_0)\int_{0}^{+\infty} \int_{0}^{+\infty}\left((\X^{1}\xi^\st)(\tilde{\Phi}^{\tau+\varsigma}_{0}
(\tilde{z}^+_{0}))\right.\\
&\qquad\qquad\qquad\qquad\qquad\qquad\qquad\qquad\left.
  -(\X^{1}\xi^\st)(\tilde{\Phi}^{\tau+\varsigma}_{0}(\tilde{z}_{0}))\right)d\varsigma d\tau\\
 &+O\left(\eps^{1+\varrho}\right).
\end{split}\end{equation}

We will now write the  double integrals in \eqref{eqn:double_integral} in a simpler form.
We show the details of the computation from the first double integral that appears in \eqref{eqn:double_integral},
as the second double integral can be treated in a similar fashion.
Denote by $\mathscr{I}^\st$ the following improper integral
\begin{equation}\label{eqn:antiderivative}
\begin{split}\mathscr{I}^\st(\tau)=&-\int_{\tau}^{+\infty}((\X^{1}H_0)(\tPhi^{\upsilon}_{0}
(\tz^+_{0}))  -(\X^{1}H_0)(\tPhi^{\upsilon}_{0}(\tz_{0})))  d\upsilon.
\end{split}\end{equation}
Since $(\X^{1}H_0)(\tPhi^{\upsilon}_{0}(\tz^+_{0}))  -(\X^{1}H_0)(\tPhi^{\upsilon}_{0}(\tz_{0})))$ approaches $0$ exponentially
as $\upsilon\to+\infty$, the above integral is convergent and moreover
\[\frac{d}{d\tau}(\mathscr{I}^\st(\tau))=(\X^{1}H_0)(\tPhi^{\tau}_{0}
(\tz^+_{0}))  -(\X^{1}H_0)(\tPhi^{\tau}_{0}(\tz_{0})).\]
That is, $\mathscr{I}^\st(\tau)$ is the antiderivative of $\tau\mapsto (\X^{1}H_0)(\tPhi^{\tau}_{0}
(\tz^+_{0}))  -(\X^{1}H_0)(\tPhi^{\tau}_{0}(\tz_{0}))$ satisfying the condition that it equals  $0$ at $+\infty$.

Making the change of variable $\upsilon =\sigma+\tau$ with $d\upsilon=d\sigma$ the double integral in \eqref{eqn:double_integral} becomes
\begin{equation}\label{eqn:first_int_2}
\begin{split}
    \int_{0}^{+\infty}\int_{\tau}^{+\infty}((\X^{1}H_0)(\tPhi^{\upsilon}_{0}
(\tz^+_{0})) -(\X^{1}H_0)(\tPhi^{\upsilon}_{0}(\tz_{0})))  d\upsilon d\tau=
 -\int_{0}^{+\infty}\mathscr{I}^\st(\tau)d\tau.
\end{split}
\end{equation}

Using Integration by Parts we obtain
\begin{equation}\begin{split}
-\int_{0}^{+\infty}\mathscr{I}^\st(\tau)d\tau=&-\tau\mathscr{I}^\st(\tau)\biggr| _{0} ^{+\infty}
\\&+\int_{0}^{+\infty} ((\X^{1}H_0)(\tPhi^{\tau}_{0} (\tz^+_{0}))  -(\X^{1}H_0)(\tPhi^{\tau}_{0}(\tz_{0})))\tau d\tau\\
=&\int_{0}^{+\infty} ((\X^{1}H_0)(\tPhi^{\tau}_{0}
(\tz^+_{0}))  -(\X^{1}H_0)(\tPhi^{\tau}_{0}(\tz_{0})))\tau d\tau.
\end{split}\end{equation}
In the above, the quantity $\tau\mathscr{I}^\st(\tau)$ obviously equals $0$  at $\tau=0$, and
equals  $0$ when $\tau\to +\infty$ since, by l'Hospital Rule
\[\lim_{\tau\to +\infty}\frac{\mathscr{I}^\st(\tau)}{\tau^{-1}}=\lim_{\tau\to +\infty}-\frac{(\X^{1}H_0)(\tPhi^{\tau}_{0}
(\tz^+_{0}))  -(\X^{1}H_0)(\tPhi^{\tau}_{0}(\tz_{0}))}{\tau^{-2}}=0,\]
since $(\X^{1}H_0)(\tPhi^{\tau}_{0}
(\tz^+_{0}))  -(\X^{1}H_0)(\tPhi^{\tau}_{0}(\tz_{0}))$ approaches $0$ at exponential rate as $\tau\to+\infty$.

A similar computation can be done to write the second double integral that appears in
\eqref{eqn:double_integral} as a single integral.

Thus, we obtain the following expression of $A_1$:
\begin{equation}\label{eqn:A1_final}\begin{split}
&\eps\left(\frac{\partial^2 h_{0}}{\partial I^2}
(I_{0}) \right)\left(\frac{\partial h_0}{\partial I}(I_0) \right)^{-1}\int^{+\infty}_{0} \left((\X^{1}H_0)(\tilde{\Phi}^{\tau }_{0}
(\tilde{z}^+_{0}))\right.\\
&\qquad\qquad\qquad\qquad\qquad\qquad\qquad\qquad\left.
  -(\X^{1}H_0)(\tilde{\Phi}^{\tau }_{0}(\tilde{z}_{0}))\right)\tau d\tau\\
&-\eps \left(\frac{\partial^2 h_{0}}{\partial I^2}(I_{0})\right) \left(\frac{\partial h_0}{\partial I}(I_0)) \right)^{-1} g_1(I_0)\int_{0}^{+\infty}\left((\X^{1}\xi^\st)(\tilde{\Phi}^{\tau }_{0}
(\tilde{z}^+_{0}))\right.\\
&\qquad\qquad\qquad\qquad\qquad\qquad\qquad\qquad\left.
-(\X^{1}\xi^\st)(\tilde{\Phi}^{\tau }_{0}(\tilde{z}_{0}))\right) \tau d\tau\\
&+O\left(\eps^{1+\varrho}\right).
\end{split}\end{equation}

Finally, we turn to the integral $A_{2}$ given by \eqref{eqn:A2}.
Using Lemma \ref{lem:estimates} equations \eqref{eqn:partial_g1_estimates} and \eqref{eqn:xi_estimates}, \eqref{eqn:xi_s_int}, as well as integration by parts similarly to above, we express $A_2$ as
\begin{equation}\label{eqn:A_2_final}
\begin{split}
-&\int_{0}^{+\infty}(\xi^{\st}\frac{\partial g_{1}}{\partial I})(\tPhi^{\tau}_{\varepsilon}(\tilde{z}^{+}_{\varepsilon}))
- (\xi^{\st}\frac{\partial g_{1}}{\partial I})(\tPhi^{\tau}_{\varepsilon}(\tilde{z}_{\varepsilon}))d \tau\\
=&-\left(\frac{\partial g_{1}}{\partial I}(I_0)\right)
\int_{0}^{+\infty}(\xi^\st((\tPhi^{\tau}_{\varepsilon}(\tilde{z}^{+}_{\varepsilon})))-
\xi^\st((\tPhi^{\tau}_{\varepsilon}(\tilde{z}_{\varepsilon}))))d\tau\\
&+O(\eps^{1+\varrho})\\
=&\eps\left( \frac{\partial g_{1}}{\partial I}(I_0) \right) \int_{0}^{+\infty}\int_{0}^{+\infty}\left((\X^{1}\xi^\st)(\tilde{\Phi}^{\tau+\varsigma}_{0}
(\tilde{z}^+_{0}))
  -(\X^{1}\xi^\st)(\tilde{\Phi}^{\tau+\varsigma}_{0}(\tilde{z}_{0}))\right) d\varsigma d\tau\\
  &+O(\eps^{1+\varrho}) \\
=&\eps\left( \frac{\partial g_{1}}{\partial I}(I_0)\right)
\int_{0}^{+\infty} \left((\X^{1}\xi^\st)(\tilde{\Phi}^{\tau}_{0}
(\tilde{z}^+_{0}))  -(\X^{1}\xi^\st)(\tilde{\Phi}^{\tau}_{0}(\tilde{z}_{0}))\right) \tau d\tau \\ &+O(\eps^{1+\varrho})
\end{split}
\end{equation}

In the above, we have used that \[(\X^0+\eps \X^1)\xi^s=\X^0\xi^s+\eps \X^1 \xi^s=\{\xi^s, H_0\}+\eps \X^1 \xi^s=\eps \X^1 \xi^s+O(\eps^2).\]

Combining \eqref{eqn:A1_final} and \eqref{eqn:A_2_final} we obtain
\begin{equation}\label{eqn:PCRBP_angle_diff_st}
\begin{split}
    \theta^\st\left(\tilde{z}^{+}_{\eps}\right)-&\theta^\st\left( \tilde{z}_{\eps}\right)\\
    =&\eps\left(\frac{\partial^2 h_{0}}{\partial I^2}
(I_{0}) \right)\left(\frac{\partial h_0}{\partial I}(I_0) \right)^{-1}\int^{+\infty}_{0} \left((\X^{1}H_0)(\tilde{\Phi}^{\tau }_{0}
(\tilde{z}^+_{0}))\right.\\
&\qquad\qquad\qquad\qquad\qquad\qquad\qquad\qquad\left.
  -(\X^{1}H_0)(\tilde{\Phi}^{\tau }_{0}(\tilde{z}_{0}))\right)\tau d\tau\\
&-\eps \left(\frac{\partial^2 h_{0}}{\partial I^2}(I_{0})\right) \left(\frac{\partial h_0}{\partial I}(I_0)) \right)^{-1} g_1(I_0)\int_{0}^{+\infty}\left((\X^{1}\xi^\st)(\tilde{\Phi}^{\tau }_{0}
(\tilde{z}^+_{0}))\right.\\
&\qquad\qquad\qquad\qquad\qquad\qquad\qquad\qquad\left.
-(\X^{1}\xi^\st)(\tilde{\Phi}^{\tau }_{0}(\tilde{z}_{0}))\right) \tau d\tau\\
  &+ \eps \left(\frac{\partial g_{1}}{\partial I}(I_0)\right)\int_{0}^{+\infty} \left((\X^{1}\xi^\st)(\tilde{\Phi}^{\tau}_{0}
(\tilde{z}^+_{0}))  -(\X^{1}\xi^\st)(\tilde{\Phi}^{\tau}_{0}(\tilde{z}_{0}))\right) \tau d\tau \\ &+O(\eps^{2})
\end{split}\end{equation}
In the above, we have also  replaced the error term $O(\eps^{1+\varrho})$ by $O(\eps^2)$, by calling the same argument as in the proof of Proposition \ref{prop:PCRTBP_change_in_I}.

A similar computation yields:
\begin{equation}\label{eqn:PCRBP_angle_diff_un}
\begin{split}
    \theta^\un\left(\tilde{z}^{-}_{\eps}\right)-&\theta^\un\left( \tilde{z}_{\eps}\right)\\
    =&-\eps\left(\frac{\partial^2 h_{0}}{\partial  I ^2}
(I_{0}) \right)\left(\frac{\partial h_0}{\partial I}(I_0) \right)^{-1}\int_{-\infty}^{0} \left((\X^{1}H_0)(\tilde{\Phi}^{\tau }_{0}
(\tilde{z}^-_{0}))\right.\\
&\qquad\qquad\qquad\qquad\qquad\qquad\qquad\qquad\left.
  -(\X^{1}H_0)(\tilde{\Phi}^{\tau }_{0}(\tilde{z}_{0}))\right)\tau d\tau\\
&+\eps \left(\frac{\partial^2 h_{0}}{\partial I^2}(I_{0})\right) \left(\frac{\partial h_0}{\partial I}(I_0)) \right)^{-1} g_1(I_0)\int^{0}_{-\infty}\left((\X^{1}\xi^\un)(\tilde{\Phi}^{\tau }_{0}
(\tilde{z}^-_{0}))\right.\\
&\qquad\qquad\qquad\qquad\qquad\qquad\qquad\qquad\left.
-(\X^{1}\xi^\un)(\tilde{\Phi}^{\tau }_{0}(\tilde{z}_{0}))\right) \tau d\tau\\
  &- \eps \left(\frac{\partial g_{1}}{\partial I}(I_0)\right) \int^{0}_{-\infty} \left((\X^{1}\xi^\un)(\tilde{\Phi}^{\tau}_{0}
(\tilde{z}^-_{0}))  -(\X^{1}\xi^\un)(\tilde{\Phi}^{\tau}_{0}(\tilde{z}_{0}))\right) \tau d\tau \\ &+O(\eps^{2})
\end{split}\end{equation}

Subtracting \eqref{eqn:PCRBP_angle_diff_un} from \eqref{eqn:PCRBP_angle_diff_st} yields
$\theta^\st\left(\tilde{z}^{+}_{\eps}\right)-\theta^\un\left(\tilde{z}^{-}_{\eps}\right)= \theta^\st\left( \tilde{z}_{\eps}\right)-\theta^\un\left( \tilde{z}_{\eps}\right)$, plus an epsilon order term consisting of the  sum of six integrals, plus an error term of order $O(\eps^2)$.
Using the notation \eqref{eqn:J}, four of these integrals are
\begin{equation}\label{eqn:final-cancellation}
  \begin{split}
    &-\eps \left(\frac{\partial^2 h_{0}}{\partial I^2}(I_{0})\right) \left(\frac{\partial h_0}{\partial I}(I_0)) \right)^{-1} g_1(I_0)J^+,\\
    &-\eps \left(\frac{\partial^2 h_{0}}{\partial I^2}(I_{0})\right) \left(\frac{\partial h_0}{\partial I}(I_0)) \right)^{-1} g_1(I_0)J^-,\\
    &\eps \left(\frac{\partial g_{1}}{\partial I}(I_0)\right)J^+,\\
    &\eps \left(\frac{\partial g_{1}}{\partial I}(I_0)\right)J^-.
  \end{split}
\end{equation}

By Corollary \ref{cor:cancellation}, since $J^++J^-=0$, the sum of the first two expressions in \eqref{eqn:final-cancellation} equals~$0$, and the sum of the last two expressions in \eqref{eqn:final-cancellation} equals~$0$.

Also, we recall from Section \ref{sec:peristence_transverse}, that for a  given unperturbed homoclinic point $\tz_0$ we selected a perturbed homoclinic point
$\tz_\eps$ satisfying condition \textbf{(A-iii-b)} \eqref{eqn:theta_st_minus_theta_un},
that is $\theta^\st(\tz_\eps)-\theta^\un(\tz_\eps)=\theta^\st(\tz_0)-\theta^\un(\tz_0)=\Delta(I(\tz_0))$.

Combining these results we obtain \eqref{eqn:PCRTBP_change_in_theta}.
\end{proof}

\section*{Acknowledgements} We are grateful to Rodney Anderson and Angel Jorba for discussions and comments.

\appendix
\section{Normally hyperbolic invariant manifolds}\label{sec:NHIM}
We briefly recall the notion of a normally hyperbolic invariant manifold.

\begin{defn}\label{defn:NHIM} Let $M$ be a $\mathcal{C}^r$-smooth manifold, $\Phi^t$ a $\mathcal{C}^r$-flow on $M$. A submanifold  (with or without boundary) $\Lambda$ of $M$ is a  normally hyperbolic invariant manifold (NHIM) for $\Phi^t$ if  it is invariant under $\Phi^t$, and there exists a splitting of the tangent bundle of $TM$ into sub-bundles over $\Lambda$
\begin{equation}
\label{eqn:NHIM_splitting}
T_z M=E^{\un}_z \oplus E^{\st}_z \oplus T_z \Lambda, \quad \forall z \in \Lambda
\end{equation}
that are invariant under $D\Phi^t$ for all $t\in\mathbb{R}$, and there exist  rates
\[\lambda_-\le \lambda_+<\lambda_c<0<\mu_c<\mu_-\le \mu_+\]
and a constant ${C}>0$, such that for all $x\in\Lambda$ we have
\begin{equation}
\label{eqn:NHIM_rates}
\begin{split} {C}e^{t\lambda_- }\|v\| \leq \|D\Phi^t(z)(v)\|\leq  {C}e^{t\lambda_+}\|v\|  \textrm{ for all } t\geq 0, &\textrm{ if and only if } v\in E^{\st}_z,\\
{C}e^{t\mu_+ }\|v\|\leq \|D\Phi^t(z)(v)\|\leq  {C}e^{t\mu_- }\|v\|  \textrm{ for all } t\leq 0,  &\textrm{ if and only if }v\in E^{\un}_z,\\
{C}e^{|t|\lambda_c }\|v\|\leq  \|D\Phi^t(z)(v)\|\leq  {C}e^{|t|\mu_c}\|v\| \textrm{ for all } t\in\mathbb{R}, &\textrm{ if and only if }v\in T_z\Lambda.
\end{split}
\end{equation}
\end{defn}

In the case when $\Phi^t$ is a Hamiltonian flow, the rates can be chosen so that \[\lambda_-=-\mu_+,\, \lambda_+=-\mu_-,\,\textrm{ and }\lambda_c=-\mu_c.\]

The regularity of the manifold $\Lambda$ depends on the rates
$\lambda^-$, $\lambda^+$, $\mu^-$, $\mu^+$,  $\lambda_c$, and $\mu_c$.
More precisely, $\Lambda$ is $\mathcal{C}^{\ell}$-differentiable, with $\ell\leq r-1$,   provided that
\begin{equation}\label{eqn:ratesdifferentiable}
\begin{split}
& \ell {\mu}_c  + {\lambda}_+ < 0, \\
& \ell {\lambda}_c +  {\mu}_- > 0.
\end{split}
\end{equation}

The manifold  $\Lambda$ has associated  unstable and stable manifolds of $\Lambda$,
denoted $W^{\un}(\Lambda)$ and $W^{\st}(\Lambda)$, respectively,
which are $\mathcal{C}^{\ell-1}$-differentiable.
They are foliated by $1$-dimensional unstable and stable manifolds (fibers) of points,
$W^{\un}(z)$, $W^{\st}(z)$, $z\in\Lambda$,  respectively, which are as smooth as the flow.

These manifolds are defined by:
\begin{equation}\label{stablemanif-flow}
\begin{split}
 W^\st(\Lambda)
 =& \{ y \,|\, d(  {\Phi}^t_\eps(y), \Lambda ) \rightarrow 0 \textrm{ as {$t \to +\infty$} }\} \\
 =& \{ y \,|\, d(  {\Phi}^t_\eps(y), \Lambda ) \le  C_y e^{t\lambda_+} ,t \ge 0\}, \\
 W^\un(\Lambda)
 =& \{ y \,|\, d(  {\Phi}^t_\eps(y), \Lambda ) \rightarrow 0 \textrm{ as {$t \to -\infty$} }\} \\
 =& \{ y \,|\, d(  {\Phi}^t_\eps(y), \Lambda ) \le  C_y e^{t\mu_- } ,t \ge 0\},\\
W^{\st}(x)
=& \{y \,|\, d( \Phi^t (y),\Phi^t (x) ) <  C_y  e^{t\lambda_+ },\, t\geq 0  \}, \\
W^{\un}(x)
=& \{y \,|\, d( \Phi^t (y), \Phi^t (x) ) <  C_y  e^{t\mu_- },\, t\leq0  \}.
\end{split}
\end{equation}

The  fibers $W^{\un}(x)$, $W^{\st}(x)$ are not invariant by the flow, but \emph{equivariant} in the sense that
\begin{equation}\label{eqn:equivariance}
\begin{split}
\Phi^t(W^\un(z))&=W^\un(\Phi^t(z)),\\
\Phi^t(W^\st(z))&=W^\st(\Phi^t(z)).
\end{split}\end{equation}\

Since $W^{\st,\un}(\Lambda)=\bigcup_{z\in\Lambda} W^{s,u}(z)$, we can define the projections along the fibers
\begin{equation}\label{eqn:projections}\begin{split}
\Omega^{+}:W^{\st}(\Lambda)\to\Lambda,\quad &\Omega^+(z)=z^+\textrm{ iff }
z \in W^{\st}(z^{+}),\\
\Omega^{-}:W^{\un}(\Lambda)\to\Lambda,\quad &\Omega^-(z)=z^-\textrm{ iff }z \in W^{\un}(z^{-}).\end{split}\end{equation}

The  point
$z^+\in \Lambda$ is characterized by
\begin{equation}\label{eqn:convergence_s}
d( {\Phi}^t(z), {\Phi}^t(z^+) ) \le
 C_z e^{t \lambda_+}, \quad \textrm { for all } t \ge
0.
\end{equation}
and the point $z^- \in \Lambda$ by
\begin{equation}\label{eqn:convergence_u}
d(  {\Phi}^t(z), {\Phi}^t(z^-) ) \le
  C_z e^{t \mu_-}, \quad \textrm { for all } t \leq 0,
\end{equation}
for some $C_z>0$.

For our applications, the most important result about NHIM's is
that they persist when we perturb the flow.  This is
the fundamental result of \cite{Fenichel71,HirschPS77,Pesin04}.
In the case when the manifold has a boundary, the persistence result
requires a step of extending the flow. This makes  that the  persistent
manifold is not invariant but only locally invariant and not unique
(it depends on the extension).

When we are given a family of flows, it is possible to choose
the extensions depending smoothly on parameters and obtain
that the manifolds depend smoothly on parameters.

The precise meaning of the smooth dependence is that the we can find
diffeomorphisms  $k_\eps: \Lambda_0 \rightarrow \Lambda_\eps$.
The maps $k_\eps(x)$ are jointly $C^r$ as functions of $x, \eps$.
The proof of this well known result is not very difficult. It suffices
to consider an extended flow
$\tilde \Phi^t(x, \eps) =  ( \Phi_\eps^t(x), \eps)$, which is
a small perturbation of $\tilde \Phi^t_0(x, \eps) =  ( \Phi_0^t(x), \eps)$.
The regularity of the NHIM of $\tilde \Phi^t$  gives the claimed
regularity of the NHIM of $\Phi^t$ with respect to parameters.

From the same proof (using the invariant objects
of the extended flow), it easily  follows the regularity with respect to
parameters of the stable and unstable bundles and the stable and unstable manifolds.

\section{Scattering map}\label{section:scattering_review}

Assume that $W^\un(\Lambda)$, $W^\st(\Lambda)$ have a   transverse intersection along a manifold $\Gamma$ satisfying:
\begin{equation}\label{goodtransversal}
\begin{split}
&T_z\Gamma = T_zW^\st(\Lambda)\cap
T_zW^\un(\Lambda), \textrm { for all } z\in\Gamma,\\
&T_zM=T_z\Gamma  \oplus
T_zW^\un(z^-)\oplus  T_zW^\st(z^+), \textrm { for all } z\in\Gamma.
\end{split}
\end{equation}

Under these conditions the projection mappings  $\Omega^\pm$ restricted to $\Gamma$ are local diffeomorphisms. We can restrict $\Gamma$ if necessary so that $\Omega^\pm$ are diffeomorphisms from $\Gamma$ onto open subsets $U^\pm$ in  $\Lambda$.

\begin{defn}\label{def:homoclinic_channel}
A homoclinic channel is a homoclinic manifold $\Gamma$ satisfying the strong transversality  condition \ref{goodtransversal}, and such that
\[\Omega^\pm_{\mid\Gamma}:\Gamma\to U^\pm:=\Omega^\pm(\Gamma)\]
is a $\mathcal{C}^{\ell-1}$-diffeomorphism.
\end{defn}

\begin{defn}\label{def:scattering map} Given a homoclinic channel $\Gamma$, the scattering map associated to $\Gamma$ is
defined as
\begin{equation*}\begin{split}
\sigma:=&\sigma^\Gamma : U^- \subseteq \Lambda \to U^+ \subseteq
\Lambda,\\
\sigma = &\Omega^+ \circ (\Omega^{-})^{-1}.
\end{split}
\end{equation*}
\end{defn}

Equivalently, $\sigma(z^-) = z^+$, provided that
$W^\un(z^-)$ intersects $W^\st(z^+)$ at a unique point
$z\in\Gamma$.

The meaning of the scattering map is that, given a homoclinic
excursion, it has two orbits in the manifold is asymptotic to.
It is asymptotic to an orbit in the past and to another
orbit in the future. The scattering map considers the future
asymptotic orbit as a function of the asymptotic in the past.
When we consider all the homoclinic orbits in a homoclinic channel
we obtain a scattering map from an open domain.
The intuition of the scattering map is that if we observe the orbit
for long times, we just measure the effect of the homoclinic  excursion
on the asymptotic behavior.   The scattering map is a very
economical way of studying these excursions since it is a map only on the
NHIM. Furthermore, as we will see now, it satisfies remarkable geometric properties.

Due to \eqref{eqn:equivariance}, the scattering map satisfies the following property
\begin{equation}
\label{eqn:invariant}
                   \Phi^T \circ\sigma ^{\Gamma}=\sigma ^{\Phi^T(\Gamma)}\circ \Phi^T
\end{equation}
for any $T\in\mathbb{R}$.

If $M$ is a  symplectic manifold, $\Phi^t$ is a Hamiltonian flow on $M$, and $\Lambda\subseteq M$ is symplectic,
then the scattering map is symplectic. If the flow is exact Hamiltonian,
the scattering map is exact symplectic. For details see \cite{DelshamsLS08a}.

In a similar fashion, we can define heteroclinic channels and associated scattering maps.

Given two NHIM's $\Lambda^1$ and $\Lambda^2$, we can define the projection mappings  $\Omega^{\pm,i}:W^{\st,\un}(\Lambda^i)\to\Lambda^i$ for $i=1,2$.
Assume that $W^\un(\Lambda^1)$ intersects transversally $W^\st(\Lambda^2)$ along a heteroclinic manifold $\Gamma$ so that:
\begin{equation}\label{goodtransversal2}
\begin{split}
&T_z\Gamma = T_zW^\un(\Lambda^1)\cap T_zW^\st(\Lambda^2), \textrm { for all } z\in\Gamma,\\
&T_zM=T_z\Gamma  \oplus
T_zW^\un(z^-)\oplus  T_zW^\st(z^+), \textrm { for all } z\in\Gamma,
\end{split}
\end{equation}
 where $z^-=\Omega^{-,1}(z) \in\Lambda^1$ and $z^+=\Omega^{+,2}(z)\in\Lambda^2$.

We can restrict $\Gamma$ so that $\Omega^{-,1}:\Gamma\to\Lambda^1$ and   $\Omega^{+,2}:\Gamma\to\Lambda^2$  are diffeomorphisms onto their corresponding images.

\begin{defn}\label{def:heteroclinic_channel}
A heteroclinic channel is a heteroclinic manifold $\Gamma$ satisfying the strong transversality  condition \ref{goodtransversal2}, and such that
\begin{equation*}\begin{split}
\Omega^{-,1}_{\mid\Gamma}:\Gamma\to U^-:=\Omega^{-,1}(\Gamma)\subseteq\Lambda^1,\\
\Omega^{+,2}_{\mid\Gamma}:\Gamma\to U^+:=\Omega^{+,2}(\Gamma)\subseteq\Lambda^2,
\end{split}\end{equation*}
are $\mathcal{C}^{l-1}$-diffeomorphisms.
\end{defn}

\begin{defn}\label{def:scattering map_heteroclinic} Given a heteroclinic channel $\Gamma$, the scattering map associated to $\Gamma$ is defined as
\begin{equation*}\begin{split}
\sigma:=&\sigma^\Gamma : U^- \subseteq \Lambda^1 \to U^+ \subseteq
\Lambda^2,\\
\sigma = &\Omega^{+,2} \circ (\Omega^{-,1})^{-1}.
\end{split}
\end{equation*}
\end{defn}

From the result of  the regularity with respect to parameters of
the stable and unstable manifolds and the fact that the scattering map is
expressed in terms of transverse intersections, we obtain that the scattering map
depends smoothly on parameters. Thus, the goal of this paper is
not to prove the  derivative of the scattering map with respect to the perturbation parameter exists, only to give explicit
formulas knowing that the derivative exists.

\section{Gronwall's inequality}\label{sec:gronwall}
In this section we apply  Gronwall's Inequality  to estimate the error between the solution of an unperturbed system
and the solution of the perturbed system, over a time of logarithmic order with respect to the size of the perturbation.

\begin{lem}\label{lem:Gronwall_application}
Consider the following differential equations:
\begin{eqnarray}
\label{eqn:eq_0}\frac{d}{dt}{z}(t)&=&\X^0(z,t)\\
\label{eqn:eq_1}\frac{d}{dt}{z}(t)&=&\X^0(z,t)+\eps \X^1(z,t,\eps)
\end{eqnarray}
Assume that $\X^0,\X^1$ are    uniformly Lipschitz continuous  in the variable $z$,  $C_0$ is the Lipschitz constant of $\X^0$,  and $\X^1$ is bounded with $\|\X^1\|\leq C_1$, for some $C_0,C_1>0$.
Let $z_0$ be a solution of the equation \eqref{eqn:eq_0} and $z_\eps$ be a solution of the equation \eqref{eqn:eq_1} such that
\begin{equation}\label{eqn:eq_3}
\|z_0(t_0)-z_\eps(t_0)\|<c\eps.
\end{equation}
Then, for  $0<\varrho_0<1$, $k\leq \frac{1-{\varrho_0}}{C_0}$, and $K=c+\frac{C_1}{C_0}$, we have
\begin{equation}\label{eqn:eq_4}
\|z_0(t)-z_\eps(t)\|< K\eps^{\varrho_0}, \textrm{ for } 0\leq t-t_0\leq k\ln(1/\eps).
\end{equation}
\end{lem}

For a proof, see \cite{gidea2021global}.

\section{Master lemmas}\label{sec:master}
In this section we recall some abstract Melnikov-type integral operators and  some of their properties
from \cite{gidea2021global}.

Consider a system as in \eqref{eqn:generalperturbation}
and the extended systems as in  \eqref{eqn:generalperturbation_t}.

Assume that, for some $\eps_1>0$, and for each $\eps\in(-\eps_1,\eps_1)$, there exists a normally hyperbolic invariant manifold $\tLambda_\eps$ for $\tPhi^\tau_\eps$, as well as a homoclinic channel  $\tGamma_\eps$, which depend  $\mathcal{C}^{\ell}$-smoothly on $\eps$. Associated to  $\tGamma_\eps$ we have projections $\Omega^\pm:\tGamma_\eps\to \Omega^\pm(\tGamma_\eps)\subseteq \tLambda_\eps$, which are local diffeomorphisms.
We are thinking of $\tPhi^\tau_\eps$, $\tLambda_\eps$, $\tGamma_\eps$  as perturbations of
$\tPhi^\tau_0$, $\tLambda_0$, $\tGamma_0$.

For $\tz_0\in\tGamma_0$ let  $\tz_\eps\in\tGamma_\eps$ be the  corresponding  homoclinic point satisfying
\eqref{eqn:theta_st_minus_theta_un}.
Because of the smooth dependence
of the normally hyperbolic manifold  and of its stable and unstable
manifolds on the perturbation, $\tz_\eps$   is $O(\eps)$-close to $\tz_0$ in the $\mathcal{C}^\ell$-topology,
that is
\begin{equation}\label{eqn:zeps_z0}
\tz_\eps=\tz_0+O(\eps).
\end{equation}

Let $(\tz_\eps,\eps)\in\widetilde{M}\mapsto {\bf F}(\tz_\eps,\eps)\in\mathbb{R}^k$ be a uniformly  $\mathcal{C}^{1}$-smooth  mapping on  $\widetilde{M}\times\R$.

We define the integral operators
\begin{equation}\label{eqn:master_operators}
\begin{split}
\mathfrak{I}^+(\bF,\tPhi^\tau_\eps, \tz_\eps)=&\int_{0}^{+\infty} \left( \mathbf{F}(\tPhi^\tau_\eps(\tz_\eps^+))-\mathbf{F}(\tPhi^\tau_\eps(\tz_\eps))\right)d\tau,\\
\mathfrak{I}^-(\bF,\tPhi^\tau_\eps,\tz_\eps)=&\int_{-\infty}^{0} \left( \mathbf{F}(\tPhi^\tau_\eps(\tz^-_\eps))-\mathbf{F}(\tPhi^\tau_\eps(\tz_\eps)\right)d\tau.
\end{split}\end{equation}

\begin{lem}[Master Lemma 1]\label{lem:master_1}
The  improper integrals \eqref{eqn:master_operators} are convergent.
The operators $\mathfrak{I}^+(\bF,\tPhi^\tau_\eps, z_\eps)$ and $\mathfrak{I}^-(\bF,\tPhi^\tau_\eps, z_\eps)$
are linear in $\bF$.
\end{lem}

\begin{lem}[Master Lemma 2]\label{lem:master_2}
\begin{equation}\label{eqn:master_differences}
\begin{split}
\bF(\tz^+_\eps)-\bF(\tz_\eps)=&-\mathfrak{I}^+((\X^0+\eps\X^1)\bF,\tPhi^\tau_\eps, \tz_\eps),\\
\bF(\tz^-_\eps)-\bF(\tz_\eps)=&\mathfrak{I}^-((\X^0+\eps\X^1)\bF,\tPhi^\tau_\eps, \tz_\eps).
\end{split}
\end{equation}
\end{lem}

\begin{lem}[Master Lemma 3]\label{lem:master_3}
\begin{equation}\label{eqn:master_gronwall}
\begin{split}
\mathfrak{I}^+(\bF,\tPhi^\tau_\eps, \tz_\eps)=&\mathfrak{I}^+(\bF,\tPhi^\tau_0, \tz_0)+O(\eps^{\varrho}),\\
\mathfrak{I}^-(\bF,\tPhi^\tau_\eps, \tz_\eps)=&\mathfrak{I}^-(\bF,\tPhi^\tau_0, \tz_0)+O(\eps^{\varrho}),
\end{split}
\end{equation}
for $0<\varrho<1$. The integrals on the right-hand side are evaluated with $\X^1=\X^1(\cdot;0)$.
\end{lem}

\begin{lem}[Master Lemma 4]\label{lem:master_4}
If $\textbf{F}=O_{\mathcal{C}^1}(\eps)$ then
\begin{equation}\label{eqn:master_gronwall}
\begin{split}
\mathfrak{I}^+(\bF,\tPhi^\tau_\eps, z_\eps)=&\mathfrak{I}^+(\bF,\tPhi^\tau_0, z_0)+O(\eps^{1+\varrho}),\\
\mathfrak{I}^-(\bF,\tPhi^\tau_\eps, z_\eps)=&\mathfrak{I}^-(\bF,\tPhi^\tau_0, z_0)+O(\eps^{1+\varrho}),
\end{split}
\end{equation}
for $0<\varrho<1$.
The integrals on the right-hand side are evaluated with $\X^1=\X^1(\cdot;0)$.
\end{lem}

The proofs of the above lemmas can be found in \cite{gidea2021global},
and similar arguments can be found in \cite{gidea2018global}.



\bibliographystyle{alpha}
\bibliography{diffusion}

\end{document}